\documentclass[a4paper,11pt]{article}
\usepackage[utf8]{inputenc}

\usepackage{verbatim}
\usepackage{amsfonts}
\usepackage{a4wide}
\usepackage{amsthm}
\usepackage{graphicx}
\usepackage{amsmath,amssymb}
\usepackage{lipsum}
\usepackage{hyperref}

\newcommand{\Hd}{\mathcal H}
\newcommand{\Lb}{\mathcal L}

\newcommand{\N}{\mathcal N}
\newcommand{\M}{\mathcal M}
\newcommand{\K}{\mathcal K}
\newcommand{\A}{\mathcal A}

\newcommand{\Ric}{\mathcal Ric}
\newcommand{\R}{\mathbb R}

\newcommand{\bv}{{\boldsymbol v}}

\newcommand{\bu}{{\boldsymbol u}}
\newcommand{\bw}{{\boldsymbol w}}
\newcommand{\bz}{{\boldsymbol z}}
\newcommand{\bX}{{\boldsymbol X}}
\newcommand{\bY}{{\boldsymbol Y}}
\newcommand{\bZ}{{\boldsymbol Z}}
\newcommand{\bW}{{\boldsymbol W}}
\newcommand{\bN}{{\boldsymbol N}}
\newcommand{\bn}{{\boldsymbol n}}

\newcommand{\bp}{{\boldsymbol p}}
\newcommand{\bP}{{\boldsymbol P}}

\newcommand{\dd}{\,\mathrm d}
\newcommand{\TV}{\mathrm{TV}}

\newcommand{\norm}{{\boldsymbol \nu}}
\newcommand{\bx}{{\boldsymbol x}}
\newcommand{\by}{{\boldsymbol y}}
\renewcommand{\div}{\mathrm{div}\,}

\newcommand{\ess}{\mathrm{ess}\,}
\newcommand{\supp}{\mathrm{supp}\,}

\newcommand{\eps}{\varepsilon}

\newcommand{\dist}{\mathrm{dist}}

\newtheorem{prop}{Proposition}
\newtheorem{thm}{Theorem}
\newtheorem{lemma}{Lemma}
\newtheorem{defn}{Definition}

\def\onedot{$\mathsurround0pt\ldotp$}
\def\cddot{
  \mathbin{\vcenter{\baselineskip.67ex
    \hbox{\onedot}\hbox{\onedot}}%
  }}%
\def\cdddot{
  \mathbin{\vcenter{\baselineskip.67ex
    \hbox{\onedot}\hbox{\onedot}\hbox{\onedot}%
  }}%
}

\newtheorem{innercustomgeneric}{\customgenericname}
\providecommand{\customgenericname}{}
\newcommand{\newcustomtheorem}[2]{%
  \newenvironment{#1}[1]
  {%
   \renewcommand\customgenericname{#2}%
   \renewcommand\theinnercustomgeneric{##1}%
   \innercustomgeneric
  }
  {\endinnercustomgeneric}
}

\newcustomtheorem{customthm}{Theorem}
\newcustomtheorem{customlemma}{Lemma}

\theoremstyle{remark}
\newtheorem*{remark}{Remark}

\author{Lorenzo Giacomelli*, Micha{\l} {\L}asica*$^\dagger$, Salvador
Moll$^\ddagger$ \\
\scriptsize *SBAI Department, Sapienza University of Rome, \\
\scriptsize Via Antonio Scarpa, 16, 00161 Roma, Italy \\
\scriptsize $^\dagger$Institute of Applied Mathematics and Mechanics, University
of Warsaw, \\
\scriptsize Banacha 2, 02-097 Warszawa, Poland \\
\scriptsize $^\ddagger$Department of Mathematical Analysis, University of
Valencia, \\
\scriptsize C/Dr. Moliner, 50, Burjassot, Spain
}
\title{Regular $1$-harmonic flow}
\date{\today}

\begin{document}
 \maketitle

\begin{abstract}
We consider the $1$-harmonic flow of maps from a bounded domain into a
submanifold of a Euclidean space, i.\,e.\;the gradient flow of the total
variation functional restricted to maps taking values in the manifold. We
restrict ourselves to Lipschitz initial data. We prove uniqueness and, in the
case of a convex domain, local existence of solutions to the flow equations.
If the target manifold has non-positive sectional curvature or in the case that
the datum is small, solutions are shown to exist globally and to become
constant in finite time. We also consider the case where the domain is a
compact Riemannian manifold without boundary, solving the homotopy
problem for $1$-harmonic maps under some assumptions.
\end{abstract}
\medskip

{\small

}

\renewcommand{\baselinestretch}{0.75}\normalsize
\tableofcontents
\renewcommand{\baselinestretch}{1}\normalsize

\section{Introduction}
Let $(\N, g)$ be a complete, connected smooth $n$-dimensional Riemannian
manifold (without boundary). Throughout the paper, without loss of generality
\cite{nash, greene}, we will treat it as an isometrically embedded submanifold
in the Euclidean space $\R^N$. Given an open, bounded
Lipschitz domain $\Omega \subset \R^m$ we consider the formal steepest descent
flow with respect to the $L^2$ distance of the functional
$\TV_\Omega^\N$: the total
variation functional constrained to functions taking values in $\N$, given
for smooth $\bu$
by
\begin{equation}
 \label{TV}
 \TV_\Omega^\N[\bu] = \int_\Omega |\nabla \bu|.
\end{equation}

Following the $L^2$-steepest descent flow is one way of controlled decreasing
$\TV_\Omega^\N$, which is a problem appearing in image processing. Besides the
case $\N\subseteq \mathbb S^{N-1}$, which appears in denoising of optical
flows \cite{TSC00} or color images \cite{TSC01}, other examples of targets
appearing in applications include the space of isometries $SO(3)\times \mathbb
R^3$ \cite{lellmann}, the cylinder $\R^2
\times \mathbb S^1$ (LCh color space) \cite{weinmanndemaretstorath} and the
space of positive definite symmetric matrices (diffusion tensors) $Sym_+(3)$
\cite{weinmanndemaretstorath}. All of these examples are homogeneous
spaces, and therefore have natural invariant metrics. Our main goal in this
paper is to develop a well-posedness theory for the flow in a generality
encompassing these cases. As some of these manifolds are non-compact, we
refrain from the unnecessary (although convenient) assumption of compactness of
$\N$.

Given a point $\bp \in \N$, we denote by
\[ \pi_\bp \colon T_{\bp} \R^N \equiv \R^N \to T_{\bp} \N\]
the orthogonal projection onto the tangent space of $\N$ at $\bp$, $T_{\bp}
\N$. Similarly, $\pi^\perp_\bp$ will denote the orthogonal projection of
$\R^N$ onto the normal space $T_{\bp} \N^\perp$. The centered dot
will denote the Euclidean scalar product on $\R^m$ or $\R^N$, while
$k$ stacked dots will denote the induced scalar product on a Cartesian
product of any $k$-tuple of these spaces. 
Calculating the first
variation of \eqref{TV} at $\bu$, one obtains that the
flow in a time interval $[0,T[$ starting with initial datum $\bu_0$ is formally
given by the system
\begin{equation}
\label{ctvflowf}
\bu_t = \pi_\bu\left( \div \tfrac{\nabla \bu}{|\nabla \bu|} \right)
\quad
\text{in }]0,T[\times \Omega,
\end{equation}
\begin{equation}
 \label{Neumannf}
 \norm^\Omega  \cdot \tfrac{\nabla \bu}{|\nabla \bu|} = 0 \quad
\text{in }]0,T[\times \partial \Omega,
\end{equation}
\begin{equation}
 \label{init}
 \bu(0, \cdot) \equiv \bu_0.
 \end{equation}

The meaning of the expression $\frac{\nabla \bu}{|\nabla \bu|}$ in
(\ref{ctvflowf}, \ref{Neumannf})
deserves a clarification even for smooth $\bu$: we understand $\frac{\nabla
\bu}{|\nabla \bu|}$ as a multifunction
\begin{equation}
 \tfrac{\nabla \bu}{|\nabla \bu|} \colon (t,\bx) \mapsto
\left\{\begin{array}{ll}
                                                       \frac{\nabla
\bu(t,\bx)}{|\nabla \bu(t,\bx)|} & \text{ if }\nabla \bu(t,\bx) \neq
\boldsymbol 0 \\
B(0,1) \subset \R^m \times T_{\bu(t,\bx)} \N & \text{ if } \nabla \bu(t,\bx) =
\boldsymbol 0
\end{array}\right.
\end{equation}
and require that (\ref{ctvflowf}, \ref{Neumannf}) are satisfied for an
appropriate selection. This is formalized in the following definition, which is
an adapted version of \cite[Definition 2.5]{acmbook}. Here and in the following
we will use the notation
\begin{equation}
 X(U, \N) = \{ \bw \in X(U, \R^N) \colon \bw(\by) \in \N
 \text{ for a.\,e. }
\by \in U\},
\end{equation}
where $U$ is any domain in $\R^l$ (or a compact $l$-dimensional Riemannian
manifold), $l=1,2,\ldots$ and $X(U, \R^N)$ is a
subspace of $L^1_{loc}(U, \R^N)$.
\begin{defn}
\label{defsolution}
 Let $T \in ]0, \infty]$. We say that
 \[\bu \in W^{1,2}_{loc}([0, T[\times \overline\Omega,
\N)\text{ with } \nabla \bu \in L^\infty_{loc}([0,T[ \times \overline\Omega,
\R^{m\cdot N})\]
is a (regular)
solution to
(\ref{ctvflowf}) (in $[0,T[$) if there exists $\bZ \in
L^\infty(]0,T[\times\Omega,
\R^{m\cdot N})$ with $\div \bZ \in L^2_{loc}([0,T[ \times \overline\Omega,
\R^N)$
satisfying
\begin{equation}
\label{inclZ}
 \bZ \in \tfrac{\nabla \bu}{|\nabla \bu|},
\end{equation}
\begin{equation}
\label{ctvflowZ}
 \bu_t = \pi_\bu( \div \bZ)
\end{equation}
$\Lb^{1+m}-\text{a.\,e. in }]0,T[\times \Omega$. We say that a regular solution $\bu$ to (\ref{ctvflowf}) satisfies (homogeneous) Neumann boundary condition (\ref{Neumannf}) if
\begin{equation}
 \label{NeumannZ}
 \norm^\Omega \cdot \bZ = \boldsymbol 0
\end{equation}
$\Lb^1 \otimes \Hd^{m-1}-\text{a.\,e. in }]0,T[\times \partial \Omega$.
\end{defn}

\begin{remark}
Due to Morrey embedding theorem, any regular solution to \eqref{ctvflowf} has a
representative that is locally H\" older continuous on
$[0,T[\times\overline{\Omega}$ \cite[Theorem 5]{hasko}. We will identify it with
this representative. In particular, the initial condition \eqref{init} can be
understood pointwise. On the other hand, $\norm^\Omega
\cdot \bZ$ in \eqref{NeumannZ} has to be understood as the normal trace of
an $L^\infty$ vector field with integrable divergence, as defined in 
\cite{temamsummable,anzellotti1}.
\end{remark}

If conditions in Definition \ref{defsolution} are satisfied, we will often
say that the pair $(\bu, \bZ)$ is a (regular) solution to
\eqref{ctvflowf} and/or \eqref{Neumannf}. We will often use equivalent (see 
e.\,g.\;the proof of Lemma \ref{bochnerlemma}) form of
\eqref{ctvflowZ}:
\begin{equation}\label{flowA}
 \bu_t = \div \bZ + \A_\bu(\bu_{x^i}, \bZ_i),
\end{equation}
where $\A_\bp$ denotes the second fundamental form of $\N$ at $\bp \in \N$.
Here and throughout
the paper, we use Einstein's summation convention.

The adjective \emph{regular} in Definition \ref{defsolution} is justified by
the following considerations. Firstly, $W^{1,\infty}(\Omega)$ is the highest
Sobolev regularity that is preserved by the scalar total variation flow
\cite{kielak, bonfortefigalli}. Secondly,
such attribute distinguishes the class of solutions in Definition
\ref{defsolution} from weak (energy) solutions, whose natural spatial regularity
is $BV(\Omega)$. However, we note that in the constrained case, even defining a
proper concept of
solution is non-trivial in the $BV$ setting, the crucial issue being an
appropriate identification of the right-hand side of \eqref{ctvflowZ} or of
\eqref{flowA}. In this regard, the only case considered so far is $\N\subseteq
\mathbb S^n$, in which \eqref{flowA} drastically simplifies due to the isotropy
of the sphere:
$$
 \bu_t = \div \bZ + \bu|\nabla \bu|.
$$
Suitably defined solutions to (\ref{ctvflowf}, \ref{Neumannf}) have been obtained in \cite{gmmn} when the initial datum is contained in an hyper-octant of $\mathbb S^n$ \cite{gmmn}. When $n=1$, the assumption on $u_0$  may be relaxed and uniqueness results are available too \cite{gmm1}. A notion of solution extending the one in \cite{gmm1,gmmn} to $(N-1)$-dimensional manifolds with unique geodesics has been proposed in \cite{dcgcurve}. Existence of solutions for a discretized Dirichlet problem for \eqref{ctvflowf} in the case $\N = \mathbb S^n$, $m=2$ has been obtained in \cite{gigakurodayamazaki}.
%
%

The validity of Definition \ref{defsolution} is supported by the well-posedness
results that we obtain. First of all, regular solutions are unique.

\begin{thm}
\label{uniqueness}
 Suppose that $\bu_1, \bu_2$ are two regular solutions to (\ref{ctvflowf},
\ref{Neumannf}) in $[0,T[$, $T \in ]0, \infty[$ such that $\bu_1(0,
\cdot) = \bu_2(0, \cdot) = \bu_0$. Then $\bu_1 \equiv \bu_2$.
%
%
\end{thm}

The proof Theorem \ref{uniqueness} is different from the proofs of analogous
results for $p$-harmonic flow in \cite{hunger1,fr1} in that we do not use
strict monotonicity of the $p$-Laplace operator (since it does not hold
for $p=1$).

Provided that $\Omega$ is convex, we are able to construct local-in-time
Lipschitz solutions to (\ref{ctvflowf}, \ref{Neumannf}). We need the assumption
of convexity, as we are forced to use global $L^p$ estimates for $\nabla\bu$.
Localization of these estimates is not available due to the strong 
degeneracy of the $1$-Laplace operator $\div \frac{\nabla \bu}{|\nabla\bu|}$. 
The assumption of
convexity is not very restrictive from the point of view of image processing, as
typical domains in applications are rectangles (or boxes of different
dimensions).

The existential theory depends on the sectional curvature $\K_\N$
of $\N$ or, equivalently, on the Riemannian curvature tensor $\mathcal R^\N$ of
$\N$. We denote by $K_\N$ the supremum of sectional curvature over $\N$, i.\,e.
\begin{equation}
\label{sectionalcurvature}
K_\N = \sup_{\bp \in \N} \max_{\bv, \bw \in T_{\bp} \N
\setminus \{0\}}
\frac{\bv \cdot \mathcal R^\N_\bp(\bv, \bw) \bw}{|\bv|^2 |\bw|^2}.
\end{equation}
Recall that $K_{SO(n) \times \mathbb R^n}$ is positive (and finite)
and $K_{\mathbb S^1 \times \R^n}$, $K_{Sym_+(n)}$ are non-positive.

\begin{thm}
\label{thmexistence}
 Suppose that
$\Omega$ is convex, the embedding of $\N$ in $\mathbb R^N$ is closed and $K_\N
< \infty$. Given $\bu_0 \in W^{1, \infty}(\Omega)$,
there exists $T = T(\N, \|\nabla \bu_0
\|_{L^\infty})$ and a regular solution $\bu$ to
(\ref{ctvflowf}, \ref{Neumannf}, \ref{init}) in $[0,T[$ satisfying the
energy inequality
\begin{equation}
\label{energyinequality}
\ess \sup_{t \in [0,T[} \int_\Omega |\nabla \bu(t, \cdot)| +
\int_0^T \int_\Omega \bu_t^2 \leq \int_\Omega |\nabla \bu_0|.
\end{equation}
\end{thm}
This theorem bears similarity to \cite[Theorem 3.4]{gigakashimayamazaki}, where
Lipschitz local-in-time solutions to \eqref{ctvflowf} are constructed in the
case where $\Omega$ is a flat torus, i.\,e.\;a box with periodic
boundary conditions. However, aside from the choice of boundary condition,
there are differences between these results --- most importantly, in
\cite{gigakashimayamazaki}, smallness of $\nabla \bu_0$ in $L^{1+\eps}(\Omega)$
is
assumed. This is because in \cite{gigakashimayamazaki}, global solutions to
$p$-harmonic flows constructed in \cite{fr2} for small intial data are used as
an approximation. In our case a different approximation scheme is proposed. In
fact we cannot use the results in \cite{fr2} as non-trivial boundary
conditions
are not handled there.

At least in the case of Dirichlet boundary data, regular solutions to \eqref{ctvflowf} can blow up in
finite time, as shown by explicit examples in \cite{DGM,giacmoll2010}. In
our case, we prove that
solutions exist globally in time,
provided that the
range of the initial datum is contained in a small enough ball in $\N$. In
fact, in this case they become constant in finite time, similarly as for
the scalar total variation flow \cite{gigakohn}. Note that in the case of
inhomogeneous Dirichlet boundary conditions, the evolution of generic initial
data under $1$-harmonic flow does not stop in finite time
\cite{gigakurodafinite}, in contrast to what is observed in the scalar
total variation flow, at least in $1$-dimensional domains
\cite{kielak}. Let us denote by $B_g(\bp, R)$ the ball centered at $\bp \in
\N$ of radius $R>0$
with respect to the metric induced by $g$ on $\N$.
\begin{thm}
\label{thmextinct}
Let $\bp_0 \in \N$, $\bu_0 \in
W^{1,\infty}(\Omega, \N)$ and $\bu$ be a regular solution to
(\ref{ctvflowf},
\ref{Neumannf}, \ref{init}) in $[0,T[$. Suppose that
$\bu_0(\Omega) \in \overline{B_g(\bp_0,R)}$, $R>0$. There exist
\begin{itemize}
 \item a constant
$R_*=R_*(\N, \bp_0) >0$ such that if $R<R_*$, then
$\bu(t,\Omega)
\in \overline{B_g(\bp_0, R)}$ for $t\in ]0,T[$,
 \item constants $\widetilde R_*= \widetilde R_*(\N, \bp_0) \in
]0,R_*[$, $C= C(\Omega, \N, \bp_0) >0$ and $\bu_* \in \N$ such that
if $R<\min\left(\widetilde R_*, \frac{T}{C}\right)$, then $\bu(t, \cdot) \equiv
\bu_*$ for $t \in ]C R,T[$.
\end{itemize}
\end{thm}

In the particular case $K_\N \leq 0$ no blow-up
occurs for any Lipschitz datum, and we can obtain
a stronger result of global existence.
\begin{thm}
\label{thmnonpositive}
 Suppose that $\Omega$ is convex and $K_\N \leq
0$. Let $\bu_0 \in
W^{1,\infty}(\Omega, \N)$. There
exists a regular
solution $\bu$ to
(\ref{ctvflowf}, \ref{Neumannf}, \ref{init}) in $[0,\infty[$ satisfying 
the energy
inequality
(\ref{energyinequality}). There exists $T_*=T_*(\bu_0)\in [0, \infty[$ and
$\bu_* = \bu_*(\bu_0)\in \N$ such that $\bu(t,\cdot) \equiv \bu_*$ for $t \geq
T_*$. Furthermore,
\[ \ess \sup_{t >0}\|\nabla \bu(t, \cdot)\|_{L^\infty(\Omega)} \leq \|\nabla
\bu_0\|_{L^\infty(\Omega)} .\]
\end{thm}
We remark that in the scalar case the preservation of the
$W^{1,\infty}$ bound follows from \cite[Corollary 5.6]{CaChNoreg}. However,
the methods there are not readily adaptable to vectorial problems.

In the present paper we
consider regular $1$-harmonic flows which are continuous, and hence capable of
generating homotopy. For this reason we find it appropriate to discuss in
detail
the case where the domain is a compact Riemannian manifold $(\M,
\gamma)$. In this
setting, the total variation functional takes form
\begin{equation}
 \TV_\M^\N[\bu] = \int_\M |\dd \bu|_\gamma \dd \mu_\gamma,
\end{equation}
where, in local coordinates $\bx \mapsto (x^1, \ldots, x^m)$ on $\M$, $|\dd
\bu|_\gamma = (\gamma^{\alpha \beta} u^i_{x^\alpha}
u^i_{x^\beta})^{\frac{1}{2}}$, we can express $\dd \mu_\gamma = |\det \,
(\gamma_{\alpha \beta})|^\frac{1}{2} \dd \Lb^m$ and \eqref{ctvflowf} takes
the form
\begin{equation}
 \label{ctvflowMf}
 \bu_t = \pi_\bu \left(\right. \div\!_\gamma \tfrac{\dd \bu}{|\dd
\bu|} \left.\right) \quad
\text{in
}]0,T[\times \M.
\end{equation}
The expression for $\div\!_\gamma$ acting on a $1$-form $\vartheta$ on $\M$ in
coordinates is
\begin{equation}
\div\!_\gamma \vartheta= |\det \,
(\gamma_{\alpha \beta})|^{-\frac{1}{2}} \left(|\det \,
(\gamma_{\alpha \beta})|^\frac{1}{2} \gamma^{\alpha \beta}
\vartheta_\beta\right)_{x^\alpha}.
\end{equation}

Observe that \eqref{ctvflowMf} is a formal limit as $p \to 1^+$ of systems
\begin{equation}
 \label{cpflowMf}
 \bu_t = \pi_\bu (\div\!_\gamma (|\dd \bu|^{p-2} \dd \bu))\quad
\text{in
}]0,T[\times \M
\end{equation}
corresponding to $p$-harmonic map flows between Riemannian manifolds. These
were first considered in the case $p = 2$ in connection with the homotopy
problem for harmonic maps, i.\,e.\;the problem of finding a harmonic map
homotopic to a given one. The problem was solved
in \cite{eellssampson} under the condition that $K_\N \leq 0$ by
constructing the harmonic
map flow. Analogous result was later obtained in \cite{fr1} for any $p>1$. We
note that there are several non-equivalent notions
of $p$-harmonic maps, among them \emph{weakly $p$-harmonic maps},
i.\,e.\;stationary weak solutions to \eqref{cpflowMf}.

We introduce the
notation
\begin{equation}
 \tfrac{\phantom{|}\dd \bu\phantom{|_\gamma}}{|\dd \bu|_\gamma} \colon (t,\bx)
\mapsto \left\{\begin{array}{ll}

\frac{\phantom{|}\dd
\bu(t,\bx)\phantom{|_\gamma}}{|\dd \bu(t,\bx)|_\gamma} & \text{ if }\dd
\bu(t,\bx)
\neq \boldsymbol 0 \\
B_\gamma(0,1) \subset T^*_{\bx} \M \times T_{\bu(t,\bx)} \N & \text{ if } \dd
\bu(t,\bx) =
\boldsymbol 0.
\end{array}\right.
\end{equation}
Measurable selections of $\frac{\phantom{|}\dd
\bu\phantom{|_\gamma}}{|\dd \bu|_\gamma}(t, \cdot)$ can be seen as $L^\infty$
sections of the bundle $T^*\M \times \R^N$ over $\M$ for a.\,e.\;$t \in ]0,T[$,
see \cite{palaisfoundations} for reference.
As in \cite{palaisfoundations}, we let
$L^p(T^*\M \times \R^N)$ denote $L^p$ sections of this bundle, $p \in [1,
\infty]$. Similarly, we denote by $L^p(]0,T[\times T^*\M \times \R^N)$ the
space of $L^p$ sections of the bundle $]0,T[ \times T^*\M \times \R^N$ over
$]0,T[ \times \M$, and by $L^p_{loc}([0,T[\times T^*\M \times \R^N)$ the space
of measurable sections of
this bundle which are
$p$-integrable
locally on $[0,T[\times\M$. We are ready to introduce a concept of
solution to \eqref{ctvflowMf}.
\begin{defn}
\label{defsolutionM}
 Let $T \in ]0, \infty]$. We say that
 \[\bu \in W^{1,2}_{loc}([0, T[\times \M,
\N)\text{ with } \dd \bu \in L^\infty_{loc}([0,T[ \times T^* \M\times
\R^N)\]  is a (regular)
solution to
(\ref{ctvflowMf}) (in $[0,T[$) if there exists $\bZ \in
L^\infty(]0,T[\times T^*\M \times \R^N)$ with $\div\!_\gamma \bZ \in
L^2_{loc}([0,T[\times \M,
\R^N)$
satisfying
\begin{equation}
\label{inclZM}
 \bZ \in \tfrac{\phantom{|}\dd \bu\phantom{|_\gamma}}{|\dd \bu|_\gamma},
\end{equation}
\begin{equation}
\label{ctvflowZM}
 \bu_t = \pi_\bu (\div\!_\gamma \bZ)
\end{equation}
$\Lb^{1+m}-\text{a.\,e. in
}]0,T[\times \M$.
\end{defn}
The
strength of our result in this case depends on the sign of Ricci curvature
$\Ric^\M$ of $\M$. We denote
\begin{equation*}
Ric_\M = \min_{\bp \in \M} \min_{\bv, \bw \in T_p \M\setminus\{\boldsymbol 0\}}
\frac{\Ric^\M_\bp(\bv, \bw)}{|\bv|_\gamma |\bw|_\gamma}.
\end{equation*}

\begin{thm}\label{homotopy}
 Let $(\M, \gamma)$ be a compact, orientable Riemannian manifold and let $(\N,
g)$ be a compact submanifold in the Euclidean space $\R^N$. Given $\bu_0 \in
W^{1,\infty}(\M,\N)$, there exists $T \in ]0,\infty]$ and a unique regular
solution to (\ref{ctvflowMf}, \ref{init}) in $[0,T[$.

If $K_\N \leq 0$, the
solution exists in $[0,\infty[$. If in addition $Ric_\M \geq 0$, there exists
a sequence $(t_k) \subset ]0,\infty[$, $t_k \to \infty$, $\bu_* \in
W^{1,\infty}(\M, \N)$ and $\bZ_* \in L^\infty(T^* \M
\times \R^N)$ with $\div\!_\gamma \bZ_* \in L^\infty(\M,
\R^N)$ such that
\begin{equation}
\label{1harmonic}
 \pi_{\bu_*} (\div\!_\gamma \bZ_*) = \boldsymbol 0, \qquad \bZ_* \in
\tfrac{\phantom{|}\dd
\bu_*\phantom{|_\gamma}}{|\dd \bu_*|_\gamma},
\end{equation}
\begin{equation}
 \bu(t_k, \cdot) \to \bu_* \text{ in } C(\M,\N).
\end{equation}
\end{thm}

As $\bu$ is continuous and the sequence $(\bu(t_k, \cdot))$ converges to
$\bu_*$ in $C(\M, \N)$, $\bu_*$ and $\bu_0$ are homotopic. Thus, we have solved
the homotopy problem for (weakly) $1$-harmonic maps assuming that $\M$ is
orientable with $Ric_\M \geq 0$ and $K_\N \leq 0$.

\bigskip
The plan of the paper is the following one: Firstly, in section
\ref{sec-unique}, we prove Theorem \ref{uniqueness}. In section \ref{sec-app},
we obtain well-posedness of an approximating system to (\ref{ctvflowf},
\ref{Neumannf}, \ref{init}) and we obtain some a priori estimates (independent
of the parameter of approximation) for their solutions. This permits us to
prove Theorem \ref{thmexistence}, to which section \ref{sec-exist} is devoted.
The asymptotic behaviour is treated in the next sections: in section
\ref{sec-extinct}, we prove Theorem \ref{thmextinct} while in section
\ref{sec-cartan}, we treat the case of nonnegative curvature; i.e Theorem
\ref{thmnonpositive}. Section \ref{sec-domain} is devoted to the case
where the domain is a compact Riemannian manifold, in which we prove Theorem
\ref{homotopy}. The last part of the paper is an appendix where some
technical lemmata are stated and proven.

\section{Uniqueness}\label{sec-unique}
In this section, we give the proof of Theorem \ref{uniqueness}.

\medskip

Let $(\bu_1, \bZ_1)$, $(\bu_2, \bZ_2)$ be two regular solutions to
(\ref{ctvflowf},
\ref{Neumannf}). For $i=1,2$ there holds
\[\bu_{i, t} = \div \bZ_i + A_{\bu_i}(\bu_{i, \bx^j}, {\bZ_{i,j}}).\]
Here and in the rest of this section, $\bu_{i,\bx^j}$ and $\bZ_{i,j}$  
denote, respectively, the derivative of $\bu_i$ in direction of $\bx^j$ and the 
$j$-th component of $\bZ_i$, $i=1,2$, $j=1,\ldots,m$. We calculate
\begin{multline}
\label{uniq1}
 \frac{1}{2} \frac{\dd}{\dd t} \int_\Omega |\bu_1 - \bu_2|^2 = \int_\Omega
(\bu_1 - \bu_2) \cdot (\div \bZ_1 - \div \bZ_2) \\ + \int_\Omega (\bu_1 -
\bu_2) \cdot (\A_{\bu_1}(\bu_{1, \bx^j}, {\bZ_{1,j}}) - \A_{\bu_2}(\bu_{2,
\bx^j}, {\bZ_{2,j}})).
\end{multline}
In the first term on the r.\,h.\,s.\,of \eqref{uniq1} we integrate by parts,
yielding
\[
\int_\Omega (\bu_1 - \bu_2) \cdot (\div \bZ_1 - \div \bZ_2) = - \int_\Omega \left(|\nabla \bu_1| - \nabla \bu_1 \cddot \bZ_2 + |\nabla \bu_2| - \nabla \bu_2 \cddot \bZ_1\right)
\]
which is non-positive as $|\bZ_i|\leq 1$, $i=1,2$. Next, we note that for
any $\bp_1, \bp_2 \in \N$ contained in a fixed compact subset $K$ of $\N$ we
have
\[
|\pi^\perp_{\bp_i}(\bp_1 - \bp_2)| \leq C_1(K)|\bp_1 - \bp_2|^2
\]
for $i=1,2$ (see e.\,g.\,\cite[Lemma A.1]{schikorra}). {As $\bu_1$, $\bu_2$
are continuous on $[0,T]\times \overline{\Omega}$} (we assume without loss of
generality that $T$
is finite), there is indeed a compact set $K=K(\bu_1, \bu_2)$ in $\N$ with
$\bu_i([0,T]\times \overline \Omega) \subset K$, $i=1,2$. Therefore, as
$A_{\bu_i}$ is valued in $T_{\bu_i}\N^\perp$ ($i=1,2$) {there is a
constant $C_2$ depending on $K$ and the norms of $\nabla \bu_1$, $\nabla
\bu_2$ in $L^\infty(]0,T[\times \Omega, \R^N)$ such that
\[
\frac{1}{2} \frac{\dd}{\dd t} \int_\Omega |\bu_1 - \bu_2|^2
\leq C_2\int_\Omega |\bu_1 - \bu_2|^2
\]
for a.\,e.\;}$t \in ]0,T[$. Thus, if $\bu_1(0, \cdot) = \bu_2(0, \cdot)$, we
have $\bu_1
\equiv \bu_2$ due to Gronwall's lemma. 
%
%
\section{The approximate system}\label{sec-app}
\label{approximate}
In this section, $\Omega \subset \R^m$ is assumed to be an open,
bounded, smooth, convex domain and $0<\alpha<1$. Given $\eps >0$, $T \in 
]0,\infty]$ we consider
the approximating system for $\bu^\eps \colon [0,T[ \times\Omega \to \N$:
\begin{equation}
 \label{ctvflowapprox}
\bu^\eps_t = \pi_{\bu^\eps} \left(\div \tfrac{\nabla
\bu^\eps}{\sqrt{\eps^2 + |\nabla
\bu^\eps|^2}}\right)\quad \text{in }]0,T[\times \Omega,
 \end{equation}
\begin{equation}
\label{Neumannapprox}
\norm^\Omega \cdot \nabla \bu^\eps = \boldsymbol 0\quad \text{in }]0,T[\times
\partial \Omega,
\end{equation}
\begin{equation}
\label{initapprox}
\bu^\eps(0,\cdot) = \bu_0.
\end{equation}

Further in this section, we will
drop the index $\eps$ and denote $\bZ = \frac{\nabla
\bu}{\sqrt{\eps^2 + |\nabla
\bu|^2}}$.

We will obtain solutions to (\ref{ctvflowapprox}, \ref{Neumannapprox},
\ref{initapprox}) in parabolic H\" older spaces as defined in \cite[Chapter
I]{lsu}. Let us introduce necessary notation. Given numbers $k = 0,1,\ldots$,
$0 <\alpha <1$ we write
$C^{\frac{k+\alpha}{2}, k+\alpha}(\Omega_I, \R^N)$ for the parabolic H\" older
space of order $k+\alpha$. Similarly, we
write $\bu \in C^{\frac{k+\alpha}{2}, k+\alpha}_{loc}(\overline \Omega_I,
\R^N)$ if $\bu \in C^{\frac{k+\alpha}{2}, k+\alpha}(\Omega_K, \R^N)$ for every
interval $K$ compactly included in $I$.

\subsection{Uniform bounds}
\label{uniform}
In this subsection, we prove essential a priori estimates for $\bu \in
C^{\frac{3+\alpha}{2},3+\alpha}_{loc}(\overline{\Omega}_{[0,T[}, \mathcal N)$
solving (\ref{ctvflowapprox}, \ref{Neumannapprox}) with a given $\varepsilon,
T>0$. For
brevity, we denote
\[
v = \left( |\nabla \bu|^2 + \eps^2\right)^\frac{1}{2}, \qquad v_0 = \left(
|\nabla \bu_0|^2 + \eps^2\right)^\frac{1}{2}.
\]
{The basic {energy estimate} reflects the gradient flow structure behind
(\ref{ctvflowapprox}, \ref{Neumannapprox}).}

\begin{lemma}\label{energylemma} Let $\bu \in
C^{\frac{3+\alpha}{2},3+\alpha}_{loc}(\overline{\Omega}_{[0,T[},
\mathcal N)$ satisfy
(\ref{ctvflowapprox}, \ref{Neumannapprox}). Then
\begin{equation}
\label{bound2}
\sup_{t \in [0, T[} \int_\Omega v(t, \cdot) + \int_0^T
\int_\Omega \bu_t^2
\leq
\int_\Omega v_0.
\end{equation}
\end{lemma}

\begin{proof}
The estimate follows from the equality
\begin{equation*}
 \frac{\dd}{\dd t} \int_\Omega v = \int_\Omega \bZ
\cddot \nabla \bu_t = - \int_\Omega \bu_t^2
\end{equation*}
which holds as $\bu_t(t,\bx) \in T_{\bu(t,\bx)} \N$ for $(t,\bx) \in ]0,T[ \times\Omega$.
\end{proof}

In order to derive further uniform bounds, our main tool is the following
version of Bochner's identity (see \cite[Chapter 1.]{linbook} for the case of
harmonic maps).

\begin{lemma}
\label{bochnerlemma}
Let $\bu \in
C^{\frac{3+\alpha}{2},3+\alpha}_{loc}(\overline{\Omega}_{[0,T[},
\mathcal N)$ satisfy
\eqref{ctvflowapprox}. Then
 \begin{equation}
 \label{bochner}
\frac{1}{2}\frac{\dd}{\dd t} |\nabla \bu|^2 = \div (\nabla \bu
\cddot \nabla \bZ) - (\pi_\bu \nabla^2 \bu) \cdddot \nabla \bZ +
\bZ_i \cdot \mathcal R^\N_\bu (\bu_{x^i},
\bu_{x^j}) \bu_{x^j}.
\end{equation}
 \end{lemma}
\begin{proof}
Given $\bx \in \Omega$, we choose a local orthonormal frame
$(\bN^k)_{k=1,\ldots,N-n}$ on $\N$ around $\bu(\bx)$. Using this frame, we
express
\begin{equation}
\label{defA}
 \pi^\perp_\bp = \bN^k_\bp \otimes \bN^k_\bp, \quad
\A_\bp(\bX, \bY) =   (\bX \cdot D_{ \bp} \bN^k\, \bY)
\bN^k_\bp,
\end{equation}
where $\bX, \bY \in T_{\bp}\N$ and $D_{\bp} \bN^k \colon T_{\bp} \N \to
\R^N$ is the tangent map to $\bN^k$ at $\bp \in \N$, that is $D_{\bw} \bN^k\,
\bw_s = (\bN^k_\bw)_s$ for any $C^1$ curve $s \mapsto \bw(s) \in \N$.

First, we calculate
\begin{equation}
 \label{pitoA}
-\bN^k_\bu \otimes \bN^k_\bu \, \div \bZ = - \bN^k_\bu (\underbrace{\bN^k_\bu
\cdot \bZ_j}_{=0})_{x^j} + \bN^k_\bu ((\bN^k_\bu)_{x^j} \cdot \bZ_j) =
\A_{\bu}(\bu_{x^j},
\bZ_j)
 \end{equation}
which allows us to rewrite \eqref{ctvflowapprox} as
\begin{equation}
\label{ctvflowapproxA}
\bu_t = \div \bZ + \A_{\bu}(\bu_{x^j}, \bZ_j).
\end{equation}

Using \eqref{ctvflowapproxA}, we obtain
\begin{multline}
 \label{passo1}
\frac{1}{2}\frac{\dd}{\dd t} |\nabla \bu|^2 = \nabla \bu \cddot \nabla \div
\bZ + \nabla \bu \cddot \nabla \A_\bu(\bu_{x^i},\bZ_i)
\\ = \left(\nabla\bu\cddot \nabla \bZ_i\right)_{x^i}
- \nabla^2 \bu \cdddot \nabla \bZ +
\big(\underbrace{\bu_{x^j} \cdot \A_\bu(\bu_{x^i}, \bZ_i)}_{=0}\big)_{x^j}
- \Delta \bu \cdot \A_\bu(\bu_{x^i}, \bZ_i)
\end{multline}
where in the last line we used that $\A_\bu$ is orthogonal to
$\bu_{x^j} \in T_{\bu} \N$.

Next, again using \eqref{defA}, we perform the following
calculations:
\begin{eqnarray*}
\big(\pi^\perp_\bu \nabla^2\bu\big)\cdddot \nabla \bZ &=&
\big(\bN^k_\bu \otimes \bN^k_\bu \nabla^2\bu \big)\cdddot \nabla \bZ
= \bN^k_\bu \cdot \bu_{x^i x^j}\, \bN^k_\bu\cdot
\bZ_{i,x^j} \\ &=& \big((\underbrace{\bN^k_\bu \cdot \bu_{x^i}}_{=0})_{x^j} -
(\bN^k_\bu)_{x^j} \cdot \bu_{x^i} \big)\big(
(\underbrace{\bN^k_\bu \cdot \bZ_i}_{=0})_{x^j}- (\bN^k_\bu)_{x^j} \cdot
\bZ_i\big)\\ &=& \A_\bu(\bu_{x^i}, \bu_{x^j}) \cdot \A_\bu(\bu_{x^j},\bZ_i)
\end{eqnarray*}
and similarly
\[
\pi^\perp_\bu \Delta \bu = (\bN^k_\bu \cdot\bu_{x^jx^j})\bN^k_\bu
=
  -\A_\bu(\bu_{x^j},
\bu_{x^j}),
\]
so that
\begin{equation*}
\Delta \bu \cdot \A_\bu(\bu_{x^i}, \bZ_i)= \pi_\bu^\perp \Delta \bu
\cdot
\A_\bu(\bu_{x^i}, \bZ_i) = - A_\bu(\bu_{x^j},\bu_{x^j})\cdot
\A_\bu(\bu_{x^i}, \bZ_i).
\end{equation*}

Hence, \eqref{passo1} may be rewritten as
\begin{multline*}
\frac{1}{2}\frac{\dd}{\dd t} |\nabla \bu|^2
= \left(\nabla\bu\cddot \nabla \bZ_i\right)_{x^i} - \left(\pi_\bu \nabla^2
\bu\right)
\cdddot \nabla \bZ
\\ - \A_\bu(\bu_{x^i}, \bu_{x^j}) \cdot \A_\bu(\bu_{x^j},\bZ_i) +
\A_\bu(\bu_{x^j},\bu_{x^j})\cdot \A_\bu(\bu_{x^i}, \bZ_i).
\end{multline*}

Finally, we recall the Gauss-Codazzi equation
\begin{equation*}
 \bW \cdot \mathcal R^\N_\bp(\bX, \bY) \bZ = \A_\bp(\bY,\bZ)\cdot
\A_\bp(\bX,\bW) -
\A_\bp(\bX,\bZ)\cdot \A_\bp(\bY,\bW)
\end{equation*}
for any quadruple of vectors $\bX, \bY, \bZ, \bW \in T_{\bp} \N$, $p \in \N$,
which finishes
the proof.
\end{proof}

{We are now ready to derive uniform Lipschitz bounds.}

\begin{lemma}
 \label{bound}Let $\bu \in
C^{\frac{3+\alpha}{2},3+\alpha}_{loc}(\overline{\Omega}_{[0,T[},
\mathcal N)$ satisfy
(\ref{ctvflowapprox}-\ref{initapprox}).

\begin{itemize}
\item [(i)] If $K_\N \in ]0, \infty[$, then
\begin{equation}
 \label{blowupest}
 \left\|v(t, \cdot)\right\|_{L^\infty} \leq \frac{\left
\|v_0\right\|_{L^\infty}}{1 - t \, K_\N \left
\|v_0\right\|_{L^\infty}}
\end{equation}
for $t \in ]0,\min(T_\dagger, T)[$, where $T_\dagger := \left(K_\N \left
\|v_0\right\|_{L^\infty}\right)^{-1}$.
\item[(ii)] If $K_\N \leq 0$, then  for $0 < t < T < T_\dagger \colon\!\!\! = 
+\infty$ there holds
 \begin{equation}
 \label{bound1}
   \left\| v(t,\cdot) \right\|_{L^\infty}
\leq
\left\|v_0 \right\|_{L^\infty}.
 \end{equation}
\end{itemize}
\end{lemma}

\begin{proof}
Given a finite $p \geq 1$, using \eqref{bochner} and integrating by parts,
we calculate
\begin{multline}
\label{pnorm}
\frac{1}{p} \frac{\dd}{\dd t} \int_\Omega v^p =\frac{1}{2}
\int_\Omega v^{p-2} \nabla \bu \cddot \nabla \bu_t
\\ =
-
\int_\Omega
v^{p-2} \big(\pi_\bu \nabla^2 \bu\big)
\cdddot \nabla \bZ - (p-2) \int_\Omega v^{p-4} \nabla \bu \cddot \nabla^2 \bu
\cdot \nabla \bZ \cddot \nabla \bu \\ +\int_{\partial \Omega} v^{p-2} \nabla
\bu \cddot \nabla
\bZ \cdot
\norm^{\Omega} + \int_\Omega v^{p-3} \bu_{x^i}
\cdot \mathcal R^\N_\bu (\bu_{x^i},
\bu_{x^j}) \bu_{x^j}.
\end{multline}
We have
\[ Z^i_{j,x^k} = v^{-1}\left(
u^i_{x^j x^k} - Z^i_j (\nabla \bu_{x^k} \cddot \bZ)\right)\]
and
\[ \nabla \bZ_j \cddot \nabla \bu = v^{-1} (\nabla \bu
\cddot \nabla \bu_{x^j} - Z^i_j \,
Z^i_k \, \nabla \bu_{x^k} \cddot \nabla \bu)\]
for $i=1,\ldots,N$ and $j,k =1, \ldots, m$. Thus, we can rewrite
\begin{equation}
\label{est1}
\nabla \bu \cddot \nabla^2 \bu \cdot \nabla \bZ \cddot \nabla
\bu =  v^{-1} \nabla
\bu \cddot \nabla \bu_{x^j} \cdot (I^N_{jk} - Z^i_j \otimes Z^i_k) \cdot
\nabla \bu_{x^k} \cddot \nabla \bu
\end{equation}
(we use the notation $\boldsymbol I^l$ for the $l$-dimensional identity
matrix). On the other hand,
\begin{equation}
\label{est2}
(\pi_\bu \nabla^2 \bu) \cdddot \nabla \bZ = v^{-1} (\pi_\bu
\nabla \bu_{x^j}) \cddot
(\boldsymbol{I}^m\!\otimes\boldsymbol{I}^N - \bZ \otimes
\bZ) \cddot \nabla \bu_{x^j} .
\end{equation}
From \eqref{est1}, \eqref{est2} and the fact that $|\bZ|\leq 1|$, it is
clear that, provided $p\geq 2$, the first two terms on the r.\;h.\;s.\;of
\eqref{pnorm} are non-positive. To
treat
the remaining boundary term, we extend $\norm^\Omega$ to a normal tubular
neighbourhood of $\partial \Omega$ in such a way that it is constant in
the fibers, and calculate (at points in $\partial \Omega$)
\begin{equation*}
 \nabla \bu \cddot \nabla \bZ \cdot \norm^\Omega = \nabla \bu \cddot
\nabla (\norm^\Omega \cdot \bZ)  - \nabla u^i
\cdot \nabla \norm^\Omega \cdot \bZ^i = - v^{-1} \norm^\Omega \cdot
\A^{\partial \Omega} (\nabla u^i,
\nabla u^i)
\end{equation*}
where by $\mathcal A^{\partial \Omega}$ we denoted the second fundamental form
of
hypersurface $\partial \Omega$. As $\Omega$ is convex, $\norm^\Omega \cdot
\mathcal A^{\partial
\Omega}$ is
non-negative. This ends the proof of \eqref{bound1} in the case $K_\N \leq 0$.

Now, assume that $K_\N \in ]0,\infty[$.  By virtue of previous calculations
and \eqref{sectionalcurvature},
we have
\[\frac{\dd}{\dd t} \left(\int_\Omega v^p\right)^{\frac{1}{p}} \leq
\left(\int_\Omega v^p\right)^{\frac{1}{p} -1} K_\N \int_\Omega v^{p+1} \leq
K_\N \left(\int_\Omega v^p\right)^{\frac{1}{p}} \|v\|_{L_\infty}.\]
Passing to the limit $p \to \infty$ we obtain, at least in a weak sense,
\[\frac{\dd}{\dd t} \|v\|_{L^\infty} \leq K_\N \|v\|_{L^\infty}^2\]
which implies \eqref{blowupest}.
\end{proof}

\subsection{Existence for the approximate system}
\label{localexistence}
In order to prove existence of solutions to the approximate system we proceed 
similarly as in
\cite[Section 3.]{hunger1}. The assumption that the embedding of $\N$ into
$\R^N$ is closed enables us to construct a metric $h$ on
$\R^N$ such that $(\N,g)$ is a totally geodesic Riemannian submanifold
of $(\R^N, h)$ (see Lemma \ref{totally} in the appendix),  i.\,e.,
\begin{itemize}
 \item the restriction of $h$ to $T\N$ coincides with $g$, that is $\left.
h_{\bp} \right|_{T_{\bp}\N \times T_{\bp}\N} \equiv g_{\bp}$  for $\bp \in \N$,
 \item there is a tubular neighborhood $\mathcal T$ of $\N$ in
$\R^N$ such that the involution $\tau \colon \mathcal T \to
\mathcal T$ given by
multiplication
by $-1$ in the fibers of $ \mathcal T$ is an isometry.
\end{itemize}

The gradient flow of the unconstrained functional $\int_\Omega |\nabla
\bu|_h$
defined for any regular enough function $\bu \colon \Omega \to \R^N$ is
expressed by the system
\begin{equation}
 \label{embedapprox}
 u^i_t = \div \tfrac{\nabla u^i}{\sqrt{\eps^2 + |\nabla \bu|_h^2}} +
\tfrac{1}{\sqrt{\eps^2 + |\nabla \bu|_h^2}} \Gamma^i_{jk}(\bu) u^j_{x^l}
u^k_{x^l},
\end{equation}
\begin{equation}
 \label{embedNeumann}
 \norm^\Omega \cdot \nabla u^i = \boldsymbol 0,
\end{equation}
where $i=1, \ldots, N$ and $\Gamma^i_{jk}$ are the Christoffel symbols of
$(\R^N, h)$. As $h$ restricted to $T\N$ coincides with $g$, the system
(\ref{embedapprox}, \ref{embedNeumann}) is identical to (\ref{ctvflowapprox},
\ref{embedNeumann}) as  long as the range of $\bu$ is contained in $\N$. In
order for $C^{\frac{3+\alpha}{2}, 3+
\alpha}_{loc}(\Omega_{[0,T[}, \N)$ solutions to the system
(\ref{embedapprox}, \ref{embedNeumann}) with initial datum $\bu_0$ to exist,
the following compatibility conditions
\begin{equation}
\label{comp1}
\norm^\Omega \cdot \nabla u^i_0 = 0
\end{equation}
\begin{equation}  \label{comp2}
\norm^\Omega \cdot \nabla \left(\div \tfrac{\nabla u_0^i}{\sqrt{\eps^2 +
|\nabla \bu_0|_h^2}} +
\tfrac{1}{\sqrt{\eps^2 + |\nabla \bu_0|_h^2}} \Gamma^i_{jk}(\bu_0) u^j_{0,x^l}
u^k_{0,x^l}\right)= 0
\end{equation}
on $\partial \Omega$ for $i=1, \ldots, N$ need to be satisfied.
\begin{prop}
 \label{existapprox}
Suppose that $K_\N < \infty$ and $\alpha \in ]0,1[$. Let $\bu_0 \in C^{3 +
\alpha}(\Omega, \N)$ satisfy (\ref{comp1}, \ref{comp2}).
Then for any $\eps >0$ the system (\ref{ctvflowapprox}-\ref{initapprox}) has a 
unique solution
 \[\bu \in C^{\frac{3+\alpha}{2}, 3+
\alpha}_{loc}(\overline{\Omega}_{[0,T_\dagger[}, \N)\] with $T_\dagger=
T_\dagger(\|\nabla \bu_0\|_{L^\infty}, K_\N) \in
]0, \infty]$ defined in Lemma \ref{bound}.
\end{prop}
Note that $T_\dagger$ in Proposition \ref{existapprox} does not depend on
$\varepsilon$.

The expressions on the right hand side of \eqref{embedapprox} make sense
without assuming a priori that the range of $\bu$ is contained in $\N$. This
fact enables us to obtain a local-in-time solution using known results on
existence for
parabolic systems. For that purpose, the authors in \cite{hunger1} or in
\cite{fr1} combine a general existence result from \cite{lunardibook} with
sectoriality estimates from \cite{vesprisystem}. On the other hand, in
\cite{misawaexistence} the author employs estimates from \cite{schlag} and
\cite{misawalinear}. However, both \cite{vesprisystem} and \cite{schlag} can
only be applied to the system with Dirichlet boundary condition, or to the case
with no boundary. As we are dealing with homogeneous Neumann boundary
condition, we appeal
instead to the existence result of Acquistapace and Terreni \cite[Theorem
1.1.]{acquistaquasi} for quasilinear systems with general boundary
conditions.

To justify its applicability to our problem, let us briefly check the assumptions. In our case,
\[
\boldsymbol A_{ij}(t,\bx,\bu,\bP) = \tfrac{1}{\sqrt{\eps^2 +
|\bP|_h^2}}\left(\delta_{ij}\boldsymbol{I} - \tfrac{\bP_i}{\sqrt{\eps^2 +
|\bP|_h^2}} \otimes \tfrac{\bP_j}{\sqrt{\eps^2 +
|\bP|_h^2}}\right)
\]
defines a locally uniformly strongly elliptic operator (see e.\,g.\,\cite{adn})
and therefore satisfies assumption (0.2) from \cite{acquistaquasi}. It is easy
to check that \eqref{embedNeumann} satisfies the complementarity
condition (0.3) from \cite{acquistaquasi}, and that the
system satisfies regularity condition (0.4) from \cite{acquistaquasi}.

Thus, for any $p>m$, we obtain the existence of unique solution to
(\ref{embedapprox}, \ref{Neumannapprox}) with
initial datum $\bu_0$ in $C^{1+\frac{\alpha}{2}}([0,T_0[,
L^p(\Omega, \R^N)) \cap C^\frac{\alpha}{2}([0,T_0[, W^{2,p}(\Omega,
\R^N))$ with some $T_0>0$. We choose $p$ so that $W^{2,p}(\Omega) \subset
C^{1,\alpha}(\Omega)$. Then, we can treat the system (\ref{embedapprox},
\ref{Neumannapprox}) as a linear system with $C^{\frac{\alpha}{2},\alpha}$
coefficients and apply \cite[Theorem VII.10.1]{lsu}
to obtain $\bu \in C^{1+\frac{\alpha}{2}, 2+ \alpha}(\Omega_{[0,T_0[})$. As
long as $\bu(t, \cdot) \in C^{2+ \alpha}(\Omega, \R^N)$, we can extend the
solution via Acquistapace-Terreni theorem. Therefore, there exists a maximal
time $T_* \leq \infty$ such that
\begin{itemize}
 \item $\bu$ exists in $C^{1+\frac{\alpha}{2},
2+ \alpha}_{loc}(\overline\Omega_{[0,T_*[},\R^N)$,
 \item the norm of $\bu$ in $C^{1+\frac{\alpha}{2},
2+ \alpha}(\Omega_{[0,t[}, \R^N)$ blows up as $t \to T_*^-$ if $T_*
< \infty$.
\end{itemize}
Since $\bu\in C^{1+\frac{\alpha}{2}, 2+
\alpha}_{loc}(\overline{\Omega}_{[0,T_*[}, \R^N)$, the coefficients of
\eqref{embedapprox}, seen as a linear equation, belong to
$C^{\frac{1+\alpha}{2}, 1+ \alpha}_{loc}(\overline{\Omega}_{[0,T_*[})$.
Therefore, provided $\bu_0\in C^{3+\alpha}$ and the additional compatibility
condition \eqref{comp2} is satisfied, we may appeal once more to \cite[Theorem
VII.10.1]{lsu} and conclude that $\bu\in C^{\frac{3+\alpha}{2}, 3+
\alpha}_{loc}(\overline{\Omega}_{[0,T_*[}, \R^N)$.

We now argue that $\bu(t, \Omega)\subset \N$ for all $t \in [0, T_*[$. Suppose,
to the contrary, that there is $t \in ]0, T_*[$ with $\bu(t, \Omega) \not
\subset \N$. Let $T_\N$ be the first time instance such that $\bu(t, \Omega)
\not \subset \N$ for $T_\N< t <T_\N + \delta$ with some $\delta >0$. Possibly
diminishing $\delta$ we can assume that $\bu(t, \Omega) \subset \mathcal T$ for
$t \in [0, T_\N + \delta[$. Then $\tau \circ \bu$ is a  solution to
\eqref{embedapprox} different to $\bu$ with the same initial and boundary
conditions, thus violating
uniqueness. Therefore, $\bu(t, \Omega)\subset \N$ for all $t \in [0, T_*[$.

It remains to show that $T_* \geq T_\dagger$, where $T_\dagger$ is defined in
Lemma \ref{bound}. Suppose that $T_* < T_\dagger$. Lemma \ref{bound} yields
\begin{equation}
 \label{uniformlinfty}
  \sup_{t \in [0, T_*[} \|\nabla \bu(t, \cdot)\|_{L^\infty(\Omega)}
<
\infty.
 \end{equation}
Let now $q > \frac{m+2}{1-\alpha}$. According to \cite[Theorem
VII.10.4 and Lemma II.3.3]{lsu}, there holds $\bu \in W^{1,q}(]0,T_*[,
L^q(\Omega, \R^N)) \cap L^q(]0,T_*[,
W^{2,q}(\Omega, \R^N))$ and consequently $\nabla \bu \in C^{\frac{\alpha}{2},
\alpha}(\Omega_{[0,T_*[},
\R^{m\cdot N})$. Now, \cite[Theorem VII.10.1]{lsu} yields $\bu \in
C^{1+\frac{\alpha}{2}, 2
+ \alpha}(\Omega_{[0,T_*[}, \R^N)$, a contradiction.

\section{Local existence}\label{sec-exist}
\label{passage}
In this section we prove Theorem \ref{thmexistence}. 

\medskip
\noindent{\it Step 1.} We assume that $\Omega$ is smooth and the initial 
datum $\bu_0 \in C^{3+\alpha}(\Omega)$
satisfies
the compatibility conditions \eqref{comp1}, \eqref{comp2}. We want to pass to
the limit $\eps \to 0^+$ in (\ref{ctvflowapprox}-\ref{initapprox}). Owing to 
Lemmata \ref{energylemma} 
and \ref{bound}, we have uniform bounds on $\bu^\eps_t$ in $L^2(]0,T[\times
\Omega$) and on $\nabla \bu^\eps$ in $L^\infty(]0, T[\times\Omega)$ for any $T
<T_\dagger$. Consequently, we also have uniform bound on $\bu^\eps$ in
$C^{\frac{1}{n+1}}(]0,T[\times \Omega)$ \cite{hasko}. All these imply that we
can extract a
sequence $(\bu_k) = (\bu^{\eps_k})$ from $(\bu^\eps)$ such that
\begin{equation}
 \bu_k \to {\bu} \text{ in } C([0,T]\times\overline \Omega), \qquad \nabla
\bu_k
\rightharpoonup \nabla {\bu} \text{ in } L^2(]0,T[\times \Omega).
\end{equation}
Due to definition of $\bZ^\eps$, we have $\|\bZ^\eps\|_{L^\infty} \leq 1$,
hence
\begin{equation}
\label{Zpassage}
 \bZ_k \overset{\ast}{\rightharpoonup}  \bZ \text{ in } L^\infty(]0,T[ \times
\Omega)
\text{ with } | \bZ| \leq 1 \text{ a.\,e.\;in } ]0,T[ \times \Omega
\end{equation}
for a sequence $(\bZ_k) = (\bZ^{\eps_k})$. Furthermore, by
virtue of the strong convergence of $\bu_k$,
\begin{equation}
\label{Ztangent}
 0 = \pi^\perp_{\bu_k} \bZ_k \overset{\ast}{\rightharpoonup}
\pi^\perp_{\bu}
\bZ \text{ in } L^\infty(]0,T[\times \Omega).
\end{equation}
Next, note that due to the H\" older bound, the family $\bu^\eps$ is
contained in a compact subset of $\N$. Rewriting \eqref{ctvflowapprox} as
\begin{equation}
\label{eqnwithA}
\bu^\eps_t = \div \bZ^\eps +
\A_{\bu^\eps}(\bZ^\eps_i, \bu^\eps_{x_i}),
\end{equation}
we deduce a uniform bound on $\div \bZ^\eps$ in $L^2(]0,T[ \times
\Omega)$. By a standard div-curl reasoning,
\begin{equation}
\label{Zproduct}
 \nabla \bu_k \cddot \bZ_k \rightharpoonup \nabla  \bu \cddot
 \bZ  \text{ in } L^2(]0,T[\times \Omega).
\end{equation}
A simple calculation shows that
\begin{equation}
\label{Zgeqeps}
 \nabla \bu^\eps \cddot \bZ^\eps = \tfrac{|\nabla \bu^\eps|^2}{\sqrt{\eps^2 +
|\nabla \bu^\eps|^2}} \geq |\nabla \bu^\eps| - \tfrac{\eps}{2}.
\end{equation}
Hence, by lower semicontinuity of $|\cdot|$ with respect to weak convergence, 
we get
\begin{equation}
\label{Zgeq}
 \nabla  \bu \cddot  \bZ \geq |\nabla  \bu|.
\end{equation}
Collecting (\ref{Zpassage}, \ref{Ztangent}, \ref{Zproduct}, \ref{Zgeq}) we
obtain that $\nabla  \bu$ and  $ \bZ$ satisfy
\eqref{inclZ}. Boundedness of $\div \bZ^\eps$ in $L^2(]0,T[ \times \Omega)$
together with strong convergence of $\bu_k$ is enough to pass to the limit in
(\ref{ctvflowapprox}, \ref{Neumannapprox}), obtaining that $\nabla
\bu$ and  $ \bZ$ satisfy (\ref{ctvflowZ}, \ref{NeumannZ}).

\medskip
\noindent{\it Step 2.} Now, we relax the regularity assumption on the 
initial datum to $\bu_0 \in W^{1,\infty}(\Omega,\N)$. Take a sequence 
$(\bu_{0,j}) \subset C^\infty(\overline{\Omega}, \N)$ such that $\bu_{0,j}$ 
converges uniformly to $\bu_0$, satisfies compatibility conditions 
(\ref{comp1}, \ref{comp2}) and 
\begin{equation}
\label{convlipnorm}
\|\nabla \bu_{0, j}\|_{L^\infty} \to \|\nabla \bu_0\|_{L^\infty}. 
\end{equation}
Such a sequence is produced in Lemma \ref{approxdatum}. By the 
previous step, there 
exists a regular solution $(\bu_j, \bZ_j)$ to (\ref{ctvflowf}, \ref{Neumannf}) 
with initial datum $\bu_{0,j}$. Recall that due to the form of estimates in 
Lemmata \ref{energylemma} and \ref{bound} the norms of $\bu_{j,t}$ in 
$L^2(]0,T[\times\Omega, \R^N)$ and of $\nabla \bu_j$ in 
$L^\infty(]0,T[\times\Omega, \R^{m\cdot N})$ are controlled by $\|\nabla 
\bu_{0,j}\|_{L^\infty}$. By virtue of \eqref{convlipnorm}, this 
control is uniform with respect to $j$. Hence, we can 
extract a subsequence converging to a regular solution to (\ref{ctvflowf}, 
\ref{Neumannf}, \ref{init}) following the same argument as in the previous 
step, with $(\bu^\eps, \bZ^\eps)$ replaced by $(\bu_j, \bZ_j)$, except that now 
we have $\nabla \bu_j \cddot \bZ_j = |\nabla \bu_j|$ instead of 
\eqref{Zgeqeps}.   

\medskip 
\noindent{\it Step 3.} Next, we lift the smoothness assumption on the 
domain. A convex domain $\Omega$ can be approximated with respect to Hausdorff 
distance by smooth convex domains $\Omega_k \subset \Omega$, $k=1,2,\ldots$. 
For a proof of this result using the signed distance function of $\Omega$, see 
Lemma \ref{convexapprox} in the appendix.  
The reasoning
in the previous paragraph yields a sequence of pairs $(\bu_k, \bZ_k)$, with 
$k$-th
one satisfying (\ref{inclZ}, \ref{ctvflowZ}, \ref{NeumannZ}) in 
$]0,T[\times\Omega_k$ with initial datum $\left. \bu_0 \right|_{\Omega_k}$. The
estimates provided by Lemmata \ref{energylemma} and \ref{bound} are uniform
with respect to $k$. Hence,
we can use them as before together with a diagonal argument to extract
subsequences of $(\bu_k)$, $(\bZ_k)$ that converge on compact 
subsets of $[0,T[ \times\Omega$ to a regular solution $(\bu, \bZ)$ to 
(\ref{ctvflowf}, \ref{init}) in $]0,T[ \times
\Omega$.

Finally, we argue that the boundary condition \eqref{NeumannZ} is satisfied.
Let us fix $\varphi \in C^1(]0,T[\times\overline \Omega)$. We have
\begin{equation*}
 \int_0^T\int_{\partial \Omega} \varphi\, \norm^\Omega \cdot \bZ =
\int_0^T\int_\Omega \varphi\, \div \bZ + \nabla \varphi \cdot
\bZ,
\end{equation*}
\begin{equation*}
 0=\int_0^T\int_{\partial \Omega_k} \varphi\, \norm^{\Omega_k} \cdot \bZ_k
=
\int_0^T\int_{\Omega_k} \varphi\, \div \bZ_k + \nabla \varphi \cdot
\bZ_k.
\end{equation*}
Let us denote $f = \varphi\, \div \bZ + \nabla \varphi \cdot
\bZ$, $f_k = \varphi\, \div \bZ_k + \nabla \varphi \cdot
\bZ_k$. By virtue of Hausdorff convergence, for a given $\eps >0$, we are 
allowed to choose $K \subset \Omega$ and $k_0$ so that $\left|]0,T[
\times (\Omega \setminus K)\right| \leq \eps^2$ and $K \subset \Omega_k$ for $k 
\geq k_0$. Recalling
\eqref{eqnwithA}, we note that $\|f_k\|_{L^2(]0,T[ \times \Omega_k)}$ is
controlled in terms of norms $\|\bu_{k,t}\|_{L^2(]0,T[\times \Omega_k)}$
and $\|\nabla \bu_k\|_{L^\infty(]0,T[\times \Omega_k)}$ and hence is uniformly
bounded. We can assume that
$(f_k|_K)_{k=k_0}^\infty$ converges weakly to $f|_K$
in $L^2(]0,T[ \times K)$. Thus, we can choose $k\geq k_0$ large
enough so that $\left|\int_0^T \int_K f - f_k\right| \leq \eps$. We
estimate
\begin{eqnarray*}
\label{Neumannestimate}
 \left|\int_0^T\int_{\partial \Omega} \varphi\, \norm^\Omega \cdot
\bZ\right|&\leq & \left|\int_0^T \int_K f - f_k\right| +
\left|\int_0^T \int_{\Omega\setminus K} f\right| + \left|\int_0^T
\int_{\Omega_k\setminus K} f_k\right| \\ \displaystyle &\leq &(1 + 
\|f\|_{L^2(]0,T[
\times \Omega)} + \|f_k\|_{L^2(]0,T[ \times \Omega_k)}) \eps.
\end{eqnarray*}
As $\eps$ and $\varphi$ are arbitrary, we are done.

\qed

\section{Finite extinction time}\label{sec-extinct}
\label{asymp}
In order to prove Theorem \ref{thmextinct} we will work directly with
regular solutions to (\ref{ctvflowf}, \ref{Neumannf},
\ref{init}) in local coordinates $\bp \mapsto (p^1, \ldots, p^n)$ on
$\N$, in which \eqref{ctvflowZ} is expressed \cite{eellssampson} as
\begin{equation}
 \label{ctvflowcoord}
  u^i_t = \div Z^i +
 \Gamma^i_{jk}( \bu)  u^j_{x^l}
 Z^k_l, \quad i = 1, \ldots, n,
\end{equation}
where $\Gamma^i_{jk}$ are the Christoffel symbols of the chosen coordinate
system. For $\bp_0 \in \N$ we denote
\begin{equation}
\label{Rstar}
 R_*(\bp_0) = \min\left\{\sup\left\{R>0\colon R\leq\frac{\pi}{2}
\left[K_{B_g(\bp_0,R)}\right]^{-\frac{1}{2}}_+\right\}, 
\frac{\ell(\bp_0)}{4}\right\},
\end{equation}
where $[K_{B_g(\bp_0,R)}]_+$ is the supremum of sectional curvature over 
$B_g(\bp_0,R)$ (compare with \eqref{sectionalcurvature}) or
$+0$ if the supremum is negative, $\ell(\bp_0)$ is the infimum of lengths of
maximal closed geodesics in $\N$ passing through $\bp_0$, and $\pi$ is
the length of a circle of radius $\frac{1}{2}$. $R_*(\bp_0)$ is positive and 
lower
than both the convexity radius and the injectivity radius of $\N$ at $\bp_0$
\cite[section 6.3.2]{petersen}.

\medskip

First, we prove
\begin{lemma} \label{boundball}
 Let $\bp_0 \in \N$, $\bu_0 \in W^{1,\infty}(\Omega)$. If $
\bu_0(\Omega) \subset \overline{B_{ g}(\bp_0, R)}$ with $R \in ]0,R_*(\bp_0)[$, 
then $
\bu(t, \Omega) \subset
\overline{B_{ g}(\bp_0, R)}$, $t \in ]0,T[$.
\end{lemma}
\begin{proof}
We proceed by contradiction. Let $T_* = \inf \{ t \in [0, T[ \colon \bu(t,
\Omega) \not \subset \overline{B_g(\bp_0, R)}\}$ Due to continuity of $\bu$,
there is a $\delta >0$ such that $\bu(t, \Omega) \subset B_g(\bp_0, 
R_*(\bp_0))$ for
$t \in [0, T_* + \delta[$.  We choose on
$B_g(\bp_0, R_*(\bp_0))$ a polar coordinate system $\bp \mapsto (p^r, 
p^{\vartheta^1},
\ldots,
p^{\vartheta^{n-1}})$ centered at $\bp_0$. Due to the block diagonal form of the
metric
in these coordinates, \eqref{ctvflowcoord} for the radial coordinate $p^r$ 
takes the
form
\begin{equation}
\label{calcpolar1}
  u^r_t = \div Z^r -
\tfrac{1}{2}
 g_{\vartheta_i\vartheta_j, r}( \bu)
u^{\vartheta^i}_{x^l}
 Z^{\vartheta^j}_l.
\end{equation}
This equation is satisfied a.\,e.\;in the open set $\{(t,\bx) \in
]0,T_* + \delta[\times \Omega \colon \bu(t,\bx) \neq \bp_0\}$. Furthermore, 
there 
holds (see the proof of Corollary 2.4 in
\cite[Chapter 6]{petersen})
\begin{equation}
\label{calcpolar2}
 \left(g_{\vartheta^i\vartheta^j, r} (\bp)\right)_{i,j=1}^{n-1} \geq
\tfrac{2}{p^r}
\cos\left(\left[K_{B_g(\bp_0,R)}\right]^{\frac{1}{2}}_+\, p^r\right)
 \left(g_{\vartheta^i\vartheta^j} (\bp)\right)_{i,j=1}^{n-1} \text{ as
quadratic forms for } \bp \in  \N.
\end{equation}
Taking into account (\ref{calcpolar1}, \ref{calcpolar2},
\ref{inclZ}, \ref{NeumannZ}) and recalling that $\bu_{x^l}$ is parallel to 
$\bZ_l$ for $l=1,\ldots,m$ we
calculate
\begin{multline}
 \label{boundballineq}
 \tfrac{1}{2}\tfrac{\dd}{\dd t} \int_\Omega ( u^r - R)_+^2 =
\int_\Omega ( u^r - R)_+ u^r_t \\ \leq - \int_{\{ \bx \in
\Omega \colon  u^r(\bx) > R\}} |\nabla
u^r| - \int_\Omega \frac{( u^r
-R)_+}{ u^r}\, \left(\cos \frac{\pi}{2}\right) \,g_{\vartheta^i\vartheta^j}(
\bu)
u^{\vartheta^i}_{x^l}
 Z^{\vartheta^j}_l \leq 0.
\end{multline}
\end{proof}

Next, we recall the notion of Riemannian center of mass. Let $R<R_*(\bp_0)$, 
$\bp_0 \in \N$. We say that $\bp_c
\in \overline{B_g(\bp_0, R)}$ is a center of mass of a Radon measure $\mu$ on
$\overline{B_g(\bp_0, R)}$ if $\bp_c$ is a minimizer of the function $f_\mu
\colon\overline{B_g(\bp_0, R)} \to [0,
\infty[$ given by
\[f_\mu(\bp)= \tfrac{1}{2}\int_{\N} \dist_g(\cdot, \bp)^2\dd \mu.\]
A unique center of mass exists for any
Radon measure on $\overline{B_g(\bp_0, R)}$ and we have
\begin{equation}
\label{0f}
 0 = \dd f_\mu(\bp_c)  = \int_{\mathcal \N} \exp_{\bp_c}^{-1} \dd \mu,
\end{equation}
where we identified elements of $T^{*}_{\bp_c} \N$ and $T_{\bp_c} \N$ via
$g$ \cite[Section 1]{karcher}. For $\bp_0 \in \N$, we denote 
\begin{equation} 
\widetilde R_*(\bp_0) = \tfrac{1}{2} \inf\left\{R_*(\bp) \colon \bp \in 
B_g\left(\bp_0, R_*(\bp_0)\right)\right\}.   
\end{equation}
 
\medskip
 We are ready to state
\begin{lemma}
 Suppose that $\bu_0 \in W^{1,\infty}(\Omega)$ satisfies $\bu_0(\Omega) \subset
\overline{B_g\left(\bp_0, R\right)}$, $\bp_0 \in  \N$, $0 < 
R < \widetilde R_*(\bp_0)$. Let
$\bp_c(t)$ be the center of mass of the pushforward measure
$\mu(t) =  \bu(t,\cdot)_\# \Lb^m$ on $\overline{B_g(\bp_0, R)}$. There
exists $C=C(\Omega, \N, \bp_0)$ such that
\begin{equation}
\label{fmuest}
 \frac{\dd }{\dd t} f_{\mu}(\bp_c)\leq - CR^{\frac{2}{m}-1}
f_{\mu}(\bp_c)^{1-\frac{1}{m}}
\end{equation}
for $t>0$.
\end{lemma}
\begin{proof}
We have
\[f_{\mu(t)}(\bp_c(t)) = \tfrac{1}{2} \int_{\Omega} \dist_{
g}( \bu(t,\cdot), \bp_c(t))^2 = \tfrac{1}{2} \int_\Omega
 u^r(t, \cdot)^2, \]
where we have chosen polar cordinates centered at $\bp_c(t)$. Employing
(\ref{0f}, \ref{calcpolar1}, \ref{calcpolar2}, \ref{inclZ}, \ref{NeumannZ}) and
observing that 
$\cos\left(\left[K_{B_g\left(\bp_c,R\right)}\right]_+^\frac{1}{2}\, R\right) 
\geq 
\cos\left(\left[K_{B_g\left(\bp_c,R_*(\bp_c)\right)}\right]_+^\frac{1}{2}\, 
\tfrac{R_*(\bp_c)}{2}\right) \in \left[\tfrac{\sqrt{2}}{2}, 1\right]$,
\begin{multline}
 \label{calcpolar3}
 \tfrac{\dd}{\dd t} f_\mu(\bp_c) =  \left\langle\dd f_\mu(\bp_c),
\bp_{c,t}\right\rangle_{T^*_{\bp_c} \N, T_{\bp_c} \N}
+ \int_\Omega  u^r  u^r_t \\ \leq - \int_\Omega
|\nabla
u^r| - \cos\left(\left[K_{B_g\left(\bp_c,R\right)}\right]_+^\frac{1}{2}\, 
R\right) \int_\Omega 
g_{\vartheta^i\vartheta^j}(
\bu) u^{\vartheta^i}_{x^l}
Z^{\vartheta^j}_l \leq - \tfrac{\sqrt{2}}{2} \int_\Omega |\nabla
\bu|_g .
\end{multline}
This equation is rigorously justified by passing to the limit $R
\to 0^+$ in the weak formulation of \eqref{boundballineq} using
Lebesgue monotone convergence theorem. Now, we choose on $B(\bp_c, R_*(\bp_c))$ 
coordinate system $\bp \mapsto
\exp_{\bp_c(t)}^{-1} \bp = (p^1, \ldots, p^n)$. From \eqref{calcpolar3}
we obtain that there exists a constant $C_1 = C_1(\N, \bp_0)>0$ such
that (recall that $ p^r = \sqrt{ p^i  p^i}$)
\begin{equation}
\label{uiui}
 \frac{\dd}{\dd t} \int_\Omega  u^i  u^i \leq - C_2
\int_\Omega \sqrt{ u^i_{x^j}  u^i_{x^j}} .
\end{equation}
Finally, applying Sobolev-Poincar\'e inequality (recall \eqref{0f}):
\begin{equation}
\label{uiui2}
\left(\int_\Omega  u^i  u^i\right)^{1-\frac{1}{m}} \leq
R^{1-\frac{2}{m}} \left(\int_\Omega \left(\sqrt{ u^i
u^i}\right)^\frac{m}{m-1}\right)^{1 - \frac{1}{m}} \leq C_2 R^{1 - \frac{2}{m}}
\int_\Omega \sqrt{ u^i_{x^j}  u^i_{x^j}}
\end{equation}
with $C_2 = C_2(\Omega)>0$. Estimates (\ref{uiui}, \ref{uiui2}) add up to
\eqref{fmuest}.
\end{proof}

\medskip
\noindent{\it Proof of Theorem \ref{thmextinct}.} First of all, by Lemma 
\ref{boundball}, we obtain the bound $\bu(t,\Omega)\subset 
\overline{B_g(\bp_0,R)}$ if $\bu_0(\Omega)\subset \overline{B_g(\bp_0,R)}$ for 
$R<\widetilde R_*(\bp_0)$ and any $t\in [0,T[$. Next,
we deduce the estimate on extinction time from \eqref{fmuest} by 
solving the ordinary
differential
inequality, which yields
\[f_{\mu(t)}(\bp_c(t))^\frac{1}{m} \leq \left(f_{\mu(0)}(\bp_c(0))^\frac{1}{m}
- \tfrac{t}{C_3}
\right)_+,\]
where
\[f_{\mu(t)}(\bp_c(t)) = \int_\Omega \dist(\bu(t, \cdot), \bp_c)^2,\]
$C_3= m
C_1^{-1}C_2 R^{1-\frac{2}{m}}$. As $f_{\mu(0)}(\bp_c(0)) \leq 
\tfrac{1}{2}|\Omega|
R^2$, there is $ \bu_* \in  \N$ such that $
\bu(t, \cdot) \equiv  \bu_*$ for $t \geq C R$, where $C= m
\left(\tfrac{|\Omega|}{2}\right)^{\frac{1}{m}} C_2^{-1} C_3$.  \qed
\section{Non-positive sectional curvature of the target}\label{sec-cartan}
This section is entirely devoted to the proof of Theorem 
\ref{thmnonpositive}.

Let $T>0$ and suppose that $\Omega$ is convex and $\N$ is a complete Riemannian
manifold with $K_\N \leq 0$. In order to prove Theorem \ref{thmnonpositive}
without the assumption that there is a closed embedding of $\N$ into 
$\mathbb R^N$,
we introduce a universal cover $\gamma \colon \widetilde \N \to \N$ of $\N$
with a Riemannian manifold $(\widetilde \N, \widetilde g)$. As a
simply-connected Riemannian manifold of non-positive curvature,
$\widetilde{\N}$
is diffeomorphic to $\R^n$ via the exponential map (this is the content of
Cartan-Hadamard theorem \cite{docarmo}). In other words, there is a global
coordinate system on $\widetilde \N$, $\widetilde \bp \mapsto
\exp_{\widetilde \bp_0}^{-1}\widetilde \bp = (\widetilde p^1, \ldots,
\widetilde p^n)$. As $\Omega$ is topologically trivial, any function $\bu_0 \in
C(\Omega, \N)$ can be lifted preserving any Sobolev or H\" older regularity to
$\widetilde \bu_0 \in C(\Omega, \widetilde \N)$ such that $\bu_0 = \gamma \circ
\widetilde \bu_0$. Then, assuming that $\Omega$ and $\bu_0$ are of class
$C^{3+\alpha}$ and $\bu_0$ satisfies the compatibility conditions
(\ref{comp1},\ref{comp2}) for $i=1,\ldots,n$, we consider the system
\begin{equation}
 \label{ctvflowapproxnonpos}
 \widetilde u^{\eps,i}_t = \div \tfrac{\nabla \widetilde
u^{\eps,i}}{\sqrt{\eps^2 + |\nabla \widetilde \bu^\eps|_{\widetilde g}^2}} +
\tfrac{1}{\sqrt{\eps^2 + |\nabla \widetilde \bu^\eps|_{\widetilde g}^2}}
\widetilde \Gamma^i_{jk}(\widetilde \bu^\eps) \widetilde u^{\eps,j}_{x^l}
\widetilde u^{\eps,k}_{x^l} \text{ in } ]0,T_*[\times\Omega,
\end{equation}
\begin{equation}
 \nabla \widetilde u^{\eps, i} \cdot \norm^\Omega = 0 \text{ in }
]0,T_*[\times \partial \Omega,
\end{equation}
\begin{equation}
 \widetilde u^{\eps, i}(0, \cdot) = \widetilde u^i_0,
\end{equation}
$i = 1, \ldots, n$. This system satisfies the assumptions of the
Aquistapace-Terreni existence theorem (see subsection \ref{localexistence}),
hence
unique solution exists for some $T_*>0$. Vector lengths
$|\widetilde \bu^\eps_t|_{\widetilde g}$ and $|\nabla \widetilde
\bu^\eps|_{\widetilde g}$ are invariant under local isometries of the target
manifold, and any Riemannian manifold is locally isometric to a submanifold in a
Euclidean space. Therefore, we can repeat the proofs of Lemmata
\ref{energylemma}, \ref{bochnerlemma} and \ref{bound} performing the
computations in a
neighbourhood of any point, obtaining bounds on $\|\widetilde
\bu^\eps_t\|_{L^2(]0,T_*[\times\Omega)}$ and $\|\nabla \widetilde
\bu^\eps\|_{L^\infty(]0,T_*[\times\Omega)}$ independent on $T_*$. Reasoning
as in subsection \ref{localexistence}, the solution can be prolonged up to the
arbitrary given $T$. Then, taking $\bu^\eps = \gamma \circ \widetilde
\bu^\eps$, we obtain a solution to (\ref{ctvflowapprox}-\ref{initapprox}).
Using the uniform bounds, we pass to the limit as in section \ref{passage}
obtaining a regular solution $\bu$ to (\ref{ctvflowf}-\ref{init}) with any
$\bu_0 \in W^{1,\infty}(\Omega)$ in any convex $\Omega$.

Finally, we consider any lifting $\widetilde \bu \colon \Omega \to \widetilde
\N$ of $\bu$ with $\widetilde \bu_t \in L^2(]0, T[\times\Omega, \R^N)$,
$\nabla \widetilde \bu \in L^\infty(]0,T[\times\Omega, \R^N)$. As $R_* = +
\infty$ for $\widetilde N$, arguments from section
\ref{asymp} imply that $\widetilde \bu$ becomes constant in finite time (if we
take large enough $T$), and consequently the same holds for $\bu = \gamma \circ
\widetilde \bu$.

\section{The case where the domain is a Riemannian manifold}\label{sec-domain}
Throughout this section, we assume that $(\M,\gamma)$ is an orientable, 
compact Riemannian manifold. Our aim is to prove Theorem \ref{homotopy}.

Similarly as in section \ref{approximate}, given $\eps, T > 0$ we first
consider the following
approximate system for $\bu^\eps\colon [0,T[\times \M \to \N$:
\begin{equation}
 \label{ctvflowMapprox}
\bu^\eps_t = \pi_{\bu^\eps} \left(\div\!_\gamma \tfrac{\dd
\bu^\eps}{\sqrt{\eps^2 + |\dd
\bu^\eps|^2}}\right)\quad \text{in }]0,T[\times \M,
 \end{equation}
\begin{equation}
\label{initMapprox}
\bu^\eps(0,\cdot) = \bu_0.
\end{equation}
Again, in what follows we drop the index $\eps$ and denote
\[\bZ = \tfrac{\dd
\bu}{\sqrt{\eps^2 + |\dd
\bu|^2}}, \qquad v = \left( |\dd \bu|^2 + \eps^2\right)^\frac{1}{2}, \qquad
v_0 = \left(
|\dd \bu_0|^2 + \eps^2\right)^\frac{1}{2}\]
\begin{lemma}
 \label{boundM}
We have
\begin{equation}
\label{bound2M}
\sup_{t \in [0, T[} \int_\M v(t, \cdot) + \int_0^T
\int_\M \bu_t^2
\leq
\int_\M v_0.
\end{equation}
There exists $T_\dagger = T_\dagger(Ric_\M, K_\N,
\|v_0\|_{L^\infty}) \in ]0,\infty]$ and a non-decreasing function
\[M_{Ric_\M, K_\N,
\|v_0\|_{L^\infty}} \colon ]0, T_\dagger[ \to ]0, \infty[\]
such that for
$t \in ]0, \min(T,T_\dagger)[$ there holds
\begin{equation}
 \label{blowupestM}
 \|v(t, \cdot)\|_{L^\infty} \leq M_{Ric_\M, K_\N,
\|v_0\|_{L^\infty}}(t).
\end{equation}
If $K_\N \leq 0$, $T_\dagger = + \infty$. If moreover $Ric_\M \geq 0$, for $t
\in ]0,T[$ there holds $\|v(t, \cdot)\|_{L^\infty} \leq \|v_0\|_{L^\infty}.$
\end{lemma}
\begin{proof}
 The Bochner formula \eqref{bochner} now takes the form (at any point $\bx \in
\M$, in normal coordinates around $\bx$)
 \begin{multline}
 \label{bochnerM}
\frac{1}{2}\frac{\dd}{\dd t} |\dd \bu|_\gamma^2 = \div\!_\gamma
(\bu_{x^\alpha} \cdot
\bZ_{;x^\alpha}) - (\pi_\bu
\bu_{x^{\alpha};x^{\beta}}) \cdot \bZ_{\alpha;x^{\beta}} \\-
\Ric^\M(\dd u^i, \bZ^i)+
\bZ_\alpha \cdot \mathcal R^\N_\bu
(\bu_{x^\alpha},
\bu_{x^{\beta}}) \bu_{x^{\beta}}.
\end{multline}
In this formula, index $;x^\alpha$ ($;x^\beta$) denotes covariant derivation
on $\M$ in direction $x^\alpha$ ($x^\beta$). Except this, the only 
difference
from formula \eqref{bochner} is the term involving the Ricci tensor of $\M$ 
which
appears when changing the order of covariant derivatives, in normal
coordinates \cite{linbook}:
\[(\div\!_\gamma \bZ)_{x^\alpha} = \bZ_{\beta; x^\beta x^\alpha} =
\bZ_{\beta;x^\alpha x^\beta} - (\Ric^\M)^{\alpha\beta} \bZ_\beta.\]

We take any $p>2$ and proceed as in the proof of Lemma \ref{bound}. Now there
is no boundary term. As $\M$ is a compact and orientable, the term
\[\int_\M \div\!_\gamma
(v^{p-2} \bu_{x^\alpha} \cdot
\bZ_{;x^\alpha}) \dd \mu_\gamma\]
vanishes due to Stokes theorem. We are led to the following estimate:
\begin{multline*}
\label{pnormM}
\frac{1}{p} \frac{\dd}{\dd t} \int_\M v^p \dd\mu_\gamma \leq -\int_\M
v^{p-3} \Ric^\M(\dd u^i, \dd u^i)\dd \mu_\gamma+ \int_\M v^{p-3} \bu_{x^\alpha}
\cdot \mathcal R^\N_\bu (\bu_{x^\alpha},
\bu_{x^\beta}) \bu_{x^\beta} \dd \mu_\gamma \\ \leq - Ric_\M \int_\M
v^{p-1}\dd \mu_\gamma+ K_\N \int_\M v^{p+1}\dd \mu_\gamma.
\end{multline*}
Using H\" older inequality,
\begin{multline*}
 \frac{\dd}{\dd t} \left( \int_\M v^p \dd \mu_\gamma
\right)^\frac{1}{p}
\\ \leq \left(\int_\M v^p\dd \mu_\gamma \right)^{\frac{1}{p} -1} \left(- Ric_\M 
\,
\mu_\gamma(\M)^\frac{1}{p} \left(\int_\M v^p\dd \mu_\gamma
\right)^{1-\frac{1}{p}} +  K_\N \|v\|_{L^\infty}\int_\M
v^p \dd \mu_\gamma\right) \\ \leq - Ric_\M
\mu_\gamma(\M)^\frac{1}{p}  +  K_\N
\|v\|_{L^\infty}\left(\int_\M
v^p \dd \mu_\gamma\right)^{\frac{1}{p}}.
\end{multline*}
Passing to the limit $p \to \infty$,
\begin{equation}
\frac{\dd}{\dd t} \|v\|_{L^\infty} \leq -Ric_\M + K_\N \|v\|_{L^\infty}^2 .
\end{equation}
We let $M_{Ric_\M, K_\N,
\|v_0\|_{L^\infty}}$ be the locally existing solution to
\[ \frac{\dd M}{\dd t} = - \min(Ric_\M, 0) + \max(K_\N, 0) M^2\]
and $T_\dagger$ be the maximal time of existence of $M_{Ric_\M, K_\N,
\|v_0\|_{L^\infty}}$.
\end{proof}

\begin{prop}
 \label{existapproxM}
Let $\bu_0 \in C^{3 +
\alpha}(\M, \N)$. There exist $T_\dagger = T_\dagger(\Ric_\M, \K_\N,
\|\nabla
\bu_0\|_{L^\infty})>0$ and unique solution $\bu \in C^{\frac{3+\alpha}{2}, 3+
\alpha}_{loc}([0,T_\dagger[\times \M, \N)$ to the system
(\ref{ctvflowMapprox}, \ref{initMapprox}).
\end{prop}
\begin{proof}
Let $\bu_0 \in C^{3+\alpha}(\M,\N)$. As in \cite[section 3]{hunger1}, we
show that there exists $T>0$ and unique solution $\bu \in C^1([0,T],
C^\alpha(\M, \N)) \cap C([0,T], C^{2+\alpha}(\M,\N))$ to (\ref{ctvflowMapprox},
\ref{initMapprox}). Using linear theory \cite{lsu}, we rise regularity of the
solution to $C^{\frac{3+\alpha}{2}, 3+ \alpha}([0,T[\times \M, \N)$. Then, 
using the uniform bound
on $\dd \bu$ in $L^\infty$ from
Lemma \ref{boundM} we show that the solution can be extended to $[0,T_\dagger[$
as in the proof of Proposition \ref{existapprox}.
\end{proof}

\medskip
\noindent{\it Proof of Theorem \ref{homotopy}.}
The proof of uniqueness follows along the lines of the proof of Theorem 
\ref{uniqueness}. An important point is that integration by parts is allowed 
because $\M$ is orientable.  

Given any initial datum $\bu_0 \in W^{1,\infty}(\M, \N)$, we take an
approximating family $(\bu_0^\eps) \subset C^{\frac{3+\alpha}{2}, 3 +
\alpha}(\M, \N)$ such that $\bu_0^\eps \to \bu_0$ as $\eps \to 0^+$ in $C(\M,
\N)$ and $\|\dd \bu_0^\eps\|_{L^\infty} \to \|\dd \bu_0\|_{L^\infty}$.
Proposition \ref{existapproxM} generates a family $(\bu^\eps)$, where
$\bu^\eps$ solves \eqref{ctvflowMapprox} with initial datum
$\bu_0^\eps$. This family satisfies uniform bounds on $(\bu^\eps_t)$ in
$L^2(]0,T_\dagger[\times \M, \R^N)$ and on $(\dd \bu^\eps)$ in
$L^\infty_{loc}([0,T_\dagger[\times T^*\M\times \R^N)$. Using these bounds, we
pass to the limit as in section \ref{passage} and obtain the regular solution
$(\bu, \bZ)$ to \eqref{ctvflowMf} in $[0,T_\dagger[$. Recall that if $K_\N \leq
0$, $T_\dagger = +\infty$.

Now we assume that $K_\N \leq
0$ and $Ric_\M \geq 0$. In this case we have
\begin{equation}
 \label{asympbounds}
 \bu_t \in L^2(]0,\infty[\times\M, \R^N), \qquad \|\dd
\bu(t,\cdot)\|_{L^\infty} \leq \|\dd \bu_0\|_{L^\infty} \text{ in a.\,e.\;}t>0.
\end{equation}
Therefore, we can choose a sequence $(t_k) \subset ]0,\infty[$, $t_k \to
\infty$ such that there exists $\bu_* \in W^{1,\infty}(\M, \N)$ with
\begin{equation}
\label{asympconv}
 \bu(t_k, \cdot) \to \bu_* \text{ in } C(\M, \N), \qquad \bu_t(t_k, \cdot)
\rightharpoonup \boldsymbol 0 \text{ in } L^2(\M, \R^N),
\end{equation}
and
\begin{equation}
\label{asympeqns}
 \bu_t(t_k, \cdot) = \pi_{\bu(t_k,\cdot)} \left(\div\!_\gamma 
\bZ(t_k,\cdot)\right), \quad
\bZ(t_k, \cdot) \in \tfrac{\phantom{|}\dd \bu\phantom{|_\gamma}}{|\dd
\bu|_\gamma}(t_k, \cdot)
\quad \mu_\gamma\text{-a.\,e.\;in } \M.
\end{equation}
The first item in \eqref{asympeqns} can be rewritten as
\begin{equation}
 \bu_t(t_k, \cdot) = \div\!_\gamma \bZ(t_k,\cdot) + \gamma^{\alpha \beta}
\A_{\bu(t_k, \cdot)}( \bu_{x^\alpha}(t_k, \cdot), \bZ_\beta(t_k, \cdot)),
\end{equation}
hence \eqref{asympbounds} implies that the sequence $\div\!_\gamma \bZ(t_k,
\cdot)$ is uniformly bounded in $L^2(\M, \R^N)$. The second item in
\eqref{asympeqns} is equivalent to
\begin{multline}
\pi^\perp_{\bu(t_k, \cdot)} \bZ(t_k,
\cdot) = \boldsymbol 0, \quad |\bZ(t_k,\cdot)|_\gamma \leq 1, \\ \quad
\gamma^{\alpha\beta}\bu_{x^\alpha}(t_k, \cdot) \cdot\bZ_\beta(t_k,\cdot) = |\dd
\bu(t_k,\cdot)|_\gamma \qquad \mu_\gamma\text{-a.\,e.\;in }\M.
\end{multline}
Hence, there exists $\bZ_\star \in L^\infty(T^*\M \times
\R^N)$
satisfying $\div\!_\gamma \bZ_* \in L^\infty(\M, \R^\N)$ and (possibly
decimating the
sequence $(t_k)$)
\begin{equation}
\label{asympconvZ}
 \bZ(t_k,\cdot) \overset{\ast}{\rightharpoonup} \bZ_* \text{ in }
L^\infty(T^*\M
\times \R^N), \qquad \div\!_\gamma \bZ(t_k,\cdot) \rightharpoonup
\div\!_\gamma\bZ_* \text{ in } L^2(\M, \R^N),
\end{equation}
\begin{equation}
\label{asymplimitZ}
 \pi^\perp_{\bu_*} \bZ_* = \boldsymbol 0, \quad |\bZ_*|_\gamma \leq 1 \qquad
\mu_\gamma\text{-a.\,e.\;in }\M.
\end{equation}
Using a standard div-curl reasoning and weak-star convergence of $\bu(t_k,
\cdot)$ in $W^{1,\infty}(\M, \N)$ we also obtain
\begin{equation}
 |\dd \bu_*|_\gamma \leq \lim \inf |\dd \bu(t_k,\cdot)|_\gamma =
\gamma^{\alpha\beta}\bu_{*,x^\alpha} \cdot\bZ_{*,\beta} \leq |\dd
\bu_*|_\gamma \qquad \mu_\gamma\text{-a.\,e.\;in }\M.
\end{equation}
This together with \eqref{asymplimitZ} yields the second item of
\eqref{1harmonic}. The first item of \eqref{1harmonic} is produced by passing
to the limit in the first item of \eqref{asympeqns} using (\ref{asympconv},
\ref{asympconvZ}).
\qed

\section*{Appendix: Technical lemmata}
\begin{customlemma}{A.1}
\label{totally}
Let $(\N, g)$ be a closed embedded Riemannian submanifold in the Euclidean
space $\R^N$. There exists a Riemannian metric $h$ on $\R^N$ such that $(\N, g)$
is a
totally geodesic Riemannian submanifold
of $(\R^N, h)$.
\end{customlemma}
\begin{proof}
 Let $R>0$. As $\N$ is a closed submanifold of $\R^N$, $\N \cap
\overline{B(0,R)}$ is compact. Hence, there is a non-increasing
function $R \mapsto \delta_R \in ]0,1[$ such that
\[N_{R,\delta} = \{ \by +
\bn \colon \by \in \N \cap B(0,R) , \bn \in T_{\by} \N^\perp, |\bn|< \delta \}\]
is a tubular neighborhood of $\N \cap B(0,R)$ in $\R^N$ that does
not intersect $\N \setminus B(0,R)$ for $\delta \in
]0, \delta_R[$. Identifying $T_{\by + \bn} N_{R,\delta_R}$ with
$T_{\by} \N \times \R^{N-n}$, we define a
Riemannian metric $h^R$ on $N_{R,\delta_R}$ as follows:
\[h^R_{\bp + \bn}(\bw_1 + \bw_1', \bw_2 + \bw_2') = g_{\bp}(\bw_1, \bw_2) +
\bw_1' \cdot
\bw_2'\]
for $\bp \in \N \cap B(0,R)$, $|\bn| < \delta_R$, $\bw_1, \bw_2 \in T_{\bp}
\N$, $\bw_1', \bw_2' \in \R^{N-n}$.
Next, we define the tubular neighborhood of $\N$
\[\mathcal T = \bigcup_{k=1}^\infty N_{k, \frac{1}{2}\delta_{k+1}}\]
so that
\[\{ \R^N \setminus \overline{\mathcal T}, N_{1,\delta_1},
N_{2,\delta_2}, \ldots\}\]
is an open cover of $\R^N$. Indeed, if $\bz\notin \R^N\setminus
{\overline{\mathcal T}}$, i.e. $\bz\in {\overline{\mathcal T}}$, then letting
$k_0$ be the smallest integer bound of $|\bz|$, we have
$$
\bz \in {\overline{\mathcal T}}\cap B(0,k_0+1) \subset
\overline{\bigcup_{k=1}^{k_0+1} N_{k, \frac{1}{2}\delta_{k+1}}} =
\bigcup_{k=1}^{k_0+1} \overline{N_{k, \frac{1}{2}\delta_{k+1}}}.
$$
Here, we used the fact that $U \cap \overline{\bigcup_{k=1}^\infty A_k} \subset
\overline{\bigcup_{k=1}^\infty U \cap A_k}$ for any sequence of sets
$A_k$ and open set $U$. Hence, by definition of $k_0$, $\bz\in
{\overline{N_{k_0, \frac{1}{2}\delta_{k_0+1}}}}$. Therefore
$$
\bz=\by+\bn \quad\text{with}\quad |\bn|\le \frac12
\delta_{k_0+1}<\delta_{k_0+1} \quad\text{and}\quad \by \in
{\overline{B(0,k_0)}}\subset B(0,k_0+1),
$$
that is, $\bz\in N_{k+1,\delta_{k+1}}$.

We take a smooth partition of unity $\{\varphi_0, \varphi_1,
\varphi_2, \ldots\}$ subordinate to this cover
 (a construction of a partition of unity subordinate to
an infinite open cover can be found in \cite[Appendix C]{tu}) and
define
\[h_{\by}(\bv_1, \bv_2) = \varphi_0(\by) \bv_1 \cdot \bv_2 + \sum_{k=1}^\infty
\varphi_k(\by)
h^k_{\by}(\bv_1, \bv_2)\]
for $\by \in \R^N$. It is easy to check that $(\N, g)$ is a totally geodesic
submanifold in
$(\R^N, h)$.
\end{proof}

\begin{customlemma}{A.2}
\label{approxdatum}
 Let $\bu_0 \in W^{1, \infty}(\Omega, \N)$. There exists a family
$(\bu_{0,\eps}) \subset C^\infty(\overline \Omega,\N)$, $\eps\in ]0,\eps_0[$,
$\eps_0 >0$ such that
\begin{itemize}
 \item $\bu_{0,\eps} \to \bu_0$ in $C(\overline{\Omega}, \mathbb R^N)$ as $\eps
\to 0^+$,
 \item $\|\nabla \bu_{0,\eps}\|_{L^\infty} \to\|\nabla
\bu_0\|_{L^\infty}$ as $\eps
\to 0^+$,
 \item $\bu_{0,\eps}$ satisfy compatibility conditions (\ref{comp1},
\ref{comp2}) for $\eps \in ]0, \eps_0[$.
\end{itemize}
\end{customlemma}
\begin{proof}
 As $\partial \Omega$ is a compact smooth submanifold of $\mathbb R^m$, there
is $\eps_0' >0$ and a tubular neighbourhood of $\partial \Omega$
\[T = \{ \by + r \norm^\Omega, \by \in \partial \Omega, r \in
]-\eps_0',
\eps_0'[\}. \]
We extend $\bu_0$ to $\bw \in W^{1,\infty}(\Omega \cup T, \N)$ putting
\[\bw(\by + r \norm^\Omega(\by)) = \by\]
for $r \in [0, \eps_0'[$.
For any $\eps \in ]0, \eps_0'[$ we define
\[\Omega_\eps = \{\bx \in \Omega \colon \dist(\bx, \partial
\Omega)>\eps\}.\]
Mollifying $\bw$ as in \cite[Theorems 4.4, 4.6]{karcher} we produce a family of
maps $(\bw_\eps)_{\eps \in ]0, \eps_0[}$, $\eps_0 \in ]0, \eps_0'[$,
$\bw \in C^\infty(\overline{\Omega}, \N)$ such that $\bw_\eps \to
\bu_0$ in $C(\overline{\Omega}, \N)$ and $\|\nabla
\bw_{\eps}\|_{L^\infty} \to
\|\nabla \bu_0\|_{L^\infty}$ as $\eps \to 0^+$.

Now, let $\eta_\eps \in C^\infty(]0,\eps[, ]0, \eps[)$ satisfy the
conditions
\begin{itemize}
 \item $\eta_\eps(r) = r$ for $r \in [\frac{\eps}{2}, \eps[$,
 \item $\eta'_\eps(r) = 0$ for $r \in ]0, \frac{\eps}{4}]$,
 \item $0\leq \eta' \leq 1$.
\end{itemize}
We define $\Phi_\eps \in C^\infty(\Omega, \Omega)$ by
\begin{equation*}
 \Phi_\eps(\bx) = \left\{ \begin{array}{ll}
                             \by - \eta_\eps(r) \norm^\Omega & \text{ if }
\bx = \by - r \norm^\Omega \in \Omega \setminus \Omega_{\eps}, \\
\bx & \text{ if } \bx \in \Omega_{\eps}.
                            \end{array} \right.
\end{equation*}
It is easy to see that $\bu_{0,\eps} = \bw_{\eps} \circ
\Phi_{\eps}$ satisfies the desired conditions.
\end{proof}

\begin{customlemma}{A.3}
 \label{convexapprox}
 Let $\Omega \subset \R^m$ be open and convex. There exists a family
$(\Omega_\eps)$ of open, convex sets with smooth boundary such that
$\Omega_\eps \subset \Omega$ for $\eps \in ]0, \eps_0[$, $\eps_0 > 0$ and the 
Hausdorff distance of $\Omega_\eps$ from $\Omega$ tends to zero
as $\eps \to 0^+$.
\end{customlemma}
\begin{proof}
 Let $d$ denote the signed distance function of $\Omega$,
i.\,e.\;
\begin{equation}
 d(\bx) = \dist(\bx, \Omega) - \dist(\bx, \R^m\!\setminus\!\Omega) 
\quad\mbox{for $\bx \in \R^m$.}
\end{equation}
This function is convex and satisfies
\begin{equation}
 \label{contractdelta}
 |d(\bx) - d(\by)| \leq |\bx-\by|\quad\mbox{for $\bx, \by$ in $\R^N$.}
 \end{equation}
Let $(\varphi_\eps)_{\eps>0}$ be a standard family of
mollifying kernels such that
\begin{equation}
 \label{supportdelta}
  \supp \varphi_\eps \subset B(\boldsymbol 0,
\eps)
 \end{equation}
and denote $d_\eps = \varphi_\eps * d$. It is easy to check that $d_\eps$ is 
smooth and convex. Let us further
denote
\[\Omega_\eps = \{\bx \in \R^m \colon d_\eps(\bx) < - \eps\}.\]
As a sublevel set of a convex function, $\Omega_\eps$ is convex. Now, denote by
$r_\Omega$ the inradius of $\Omega$, equivalently $r_\Omega = \min d$. Take
$\eps_0 = \frac{r_\Omega}{3}$ and assume $\eps < \eps_0$. Suppose that $d(\bx)
\geq 0$. Due to (\ref{supportdelta}, \ref{contractdelta}), we have
\[d_\eps(\bx)=\int_{B(\bx, \eps)} \varphi_\eps(\bx-\by) d(\by) \dd
\by > -\eps.\]
Hence, $\overline{\Omega}_\eps \subset \Omega$. Similarly, if $d(\bx)
\leq -2\eps$, then $d_\eps(\bx) < -\eps$. This in turn implies that
\begin{equation}
\label{distanceeps}
\dist(\partial \Omega_\eps, \partial \Omega) = \min \{- d(\bx)\colon
d_\eps(\bx)=-\eps\} < 2 \eps .
\end{equation}
Denoting by $\bx_\Omega$ the center of any circle inscribed in
$\overline{\Omega}$,
\begin{equation}
\label{nominimum}
 \min d_\eps \leq \int_{B(\bx_\Omega, \eps)} \varphi_\eps(\bx_\Omega -
\by) d(\by) \dd \by < - r_\Omega + \eps < - 2 \eps .
\end{equation}
Recall that a critical point of a smooth convex function on $\R^m$ is
necessarily its global (possibly improper) minimum. Hence, by virtue of
(\ref{distanceeps}, \ref{nominimum}), $\Omega_\eps$ does not contain critical
points of $d_\eps$, and so it is a smooth hypersurface. Finally,
\eqref{distanceeps} implies the Hausdorff convergence of $\Omega_\eps$ to
$\Omega$ as $\eps \to 0^+$.

\end{proof}

\section*{Acknowledgments}
M.\,{\L}.\,has been supported by the grant no.\;2014/13/N/ST1/02622 of the
National Science Centre, Poland. S. M. has been partially supported by the
Spanish MINECO and FEDER project MTM2015-70227-P as well as the Simons
Foundation grant
346300 and the Polish Government MNiSW 2015-2019 matching fund.

 \bibliographystyle{plain}
 \bibliography{./aniso}
\end{document}